\documentclass[12pt,a4paper]{article}
\usepackage[pagewise]{lineno}
\usepackage[T1]{fontenc}
\usepackage[utf8]{inputenc}
\pdfoutput=1
\usepackage{lmodern}
\usepackage{textcomp}
\usepackage[DIV=12,BCOR=5mm,headinclude=false,footinclude=false]{typearea}
\usepackage[font=small,labelfont=bf]{caption}
\usepackage{titlesec}
\usepackage{tocloft}
\usepackage{amsmath,amssymb}
\usepackage{enumerate} 
\usepackage{amsfonts}
\usepackage{graphicx}
\usepackage{amsthm}
\usepackage{yfonts}
\usepackage{tikz-cd}
\usepackage{amssymb}
\usepackage{mathrsfs}
\usepackage[parfill]{parskip}
\usepackage{hyperref}
\usepackage[nameinlink]{cleveref}
\usepackage{MnSymbol }
\usepackage{verbatim}
\usepackage{array}
\titleformat{\subsection}[runin]
  {\normalfont\large\bfseries}{\thesubsection}{0em}{}
  \titleformat{\subsubsection}[runin]
  {\normalfont\normalsize\bfseries}{\thesubsubsection}{0em}{}
\newcommand\restr[2]{{
  \left.\kern-\nulldelimiterspace 
  #1
  \vphantom{\big|}
  \right|_{#2}
  }}
\newtheorem{theorem}{Theorem}

\newtheorem{corollary}{Corollary}[subsection]
\newtheorem{lemma}{Lemma}[subsection]
\newtheorem{prop}{Proposition}[section]
\theoremstyle{definition}
\newtheorem*{rem}{Remark}
\newtheorem*{definition}{Definition}

\newcommand{\C}{\mathbb{C}}

\newcommand{\fm}{\mathfrak{m}}
\newcommand{\fn}{\mathfrak{n}}
\newcommand{\gl}{\mathrm{GL}}
\newcommand{\fs}{\mathfrak{s}}
\newcommand{\Irr}{\mathfrak{Irr}}
\newcommand{\Rep}{\mathfrak{Rep}}
\newcommand{\Ho}{\mathrm{Hom}}
\newcommand{\cus}{\mathrm{cusp}}
\newcommand{\Scu}{\mathfrak{SC}}
\newcommand{\scus}{\mathrm{scusp}}
\newcommand{\ol}[1]{\overline{#1}}

\newcommand{\De}{\Delta}

\newcommand{\Ms}{\mathcal{M}ult_R}
\newcommand{\lp}{\preceq}
\newcommand{\lpp}{\prec}
\newcommand{\gp}{\succeq}
\newcommand{\ind}{\mathrm{Ind}}
\newcommand{\Ind}{\mathrm{ind}}
\newcommand{\hra}{\hookrightarrow}

\newcommand{\tfm}{{\widetilde{\fm}}}

\newcommand{\ql}{{\ol{\mathbb{Q}}_\ell}}
\newcommand{\fl}{{\ol{\mathbb{F}}_\ell}}
\newcommand{\zl}{{\ol{\mathbb{Z}}_\ell}}

\newcommand{\trho}{{\Tilde{\rho}}}
\newcommand{\rl}{\mathrm{r}_\ell}
\newcommand{\Z}{\mathrm{Z}}

\newcommand{\ain}[3]{#1\in\{#2,\ldots,#3\}}

\newcommand{\fk}{\mathfrak{k}}

\newcommand{\TD}{{\widetilde{\De}}}

\newcommand{\ZZ}{\mathbb{Z}}

\newcommand{\ra}{\rightarrow}
\newcommand{\sra}{\twoheadrightarrow}

\newcommand{\db}{{\textbf{d}}}
\newcommand{\dd}{\,\mathrm{d}_l}

\newcommand{\Cc}{\mathrm{C}}
\newcommand{\CC}{{\mathbb{C}}}

\newcommand{\scu}{\mathfrak{C}^\square}

\newcommand{\Sc}{\mathcal{S}}
\newcommand{\cc}{^\mathfrak{c}}

\newcommand{\Msl}{\mathcal{M}ult_{\fl}}
\newcommand{\Msq}{\mathcal{M}ult_{\ql}}
\newcommand{\Wc}{\mathcal{W}}
\newcommand{\tfn}{\widetilde{\fn}}

\newcommand{\abs}{\lvert-\lvert}
\newcommand{\Ff}{\mathrm{F}}
\newcommand{\of}{\mathfrak{O}_\Ff}

\newcommand{\NN}{\mathbb{N}}
\newcommand{\cusp}{\mathfrak{C}}
\newcommand{\Kr}{\mathcal{K}}

\newcommand{\tc}{\mathcal{T}}
\newcommand{\Msi}{\mathcal{M}ult_{R,\square}}

\newcommand{\MsQ}{\mathcal{M}ult(Q)}

\newcommand{\grdim}{\mathrm{grdim}}
\newcommand{\sccp}{\Sc_{gen,\psi}}
\newcommand{\sccpu}{\Sc_{gen,\psi}^\cup}
\newcommand{\Scp}{\Sc_{\psi}}
\newcommand{\Msli}{\mathcal{M}ult_{\fl,\square}}
\title{The generic extension map and modular standard modules}
\author{Johannes Droschl}
\date{}

\begin{document}
\maketitle
\begin{abstract}
    In this paper we study two classes of $\ell$-modular standard modules of the general linear group. The first class is obtained by reducing existing standard modules over $\ql$ to $\fl$ with respect to their natural integral structure. The second class is obtained by studying the generic extension map of the cyclical quiver, which was motivated by the construction of certain monomial bases of quantum algebras. In the latter case we also manage to prove a modular version of the Langlands classification, similar to the work of Langlands and Zelevinsky over $\CC$.
    We moreover compute the corresponding $\ell$-modular Rankin-Selberg $L$-function and check that they agree with the $L$-functions of their $\Cc$-parameters constructed by Kurinczuk and Matringe.
\end{abstract}
\section{Introduction}
Let $\Ff$ be a non-archimedean local field with ring of integers $\of$ and residue field $\mathrm{k}$, whose cardinality we denote by $q$, and $G$ be a reductive group over $\Ff$. We also fix a prime $\ell$ not dividing $q$ together with a field of coefficients $R\in\{\fl,\ql,\CC\}$. The classification of irreducible smooth representations of $G(\mathrm{F})$ over $R$ is of great importance in establishing the local Langlands correspondence. If $R=\CC$, such a classification has first been achieved by Langlands by what is now known as the \emph{Langlands classification}. His classification associates to each irreducible smooth representation $\pi$ of $G(\Ff)$ a parabolic subgroup $P\subseteq G$, a tempered representation $\sigma$ of the $\Ff$-points of the Levi-component $M$ of $P$ and a character $\eta$ of the center of $M$ satisfying certain inequalities.
Then $\pi$ can be realized as the unique quotient of the parabolically induced representation $\Sc(\pi)=\Ind_{P(\Ff)}^{G(\Ff)}(\sigma\otimes\eta)$, known as the standard module associated to $\pi$, and the triple $(P,\sigma,\eta)$ is unique up to certain natural symmetries. The analytical nature of the Langlands classification as sketched above makes it clear that it is a non-trivial task to extend it to fields of non-zero characteristics, especially as soon as one ventures beyond the banal setting of \cite{MinSecII}. Therefore the construction of modular standard modules remained an open problem. The proposed definitions given in this paper avoid the above issue by considering standard modules which are, in the modular case, not necessarily induced representations, but rather certain subrepresentations of the space of Whittaker functions. 

To motivate this approach we return for a moment to the problem of the classification of irreducible modular representations. We recall that in the case $G=\gl_n$ Bernstein and Zelevinsky gave in \cite{ZelII}, \cite{BerZel}, \cite{BerZelI} a second classification of irreducible complex representation in terms of combinatorial objects called \emph{multisegments}. To this date this complete second classification is a feature only present in the representation theory of $\gl_n$ and its inner forms. Certain extensions of these ideas to more general groups have been pursued \cite{HanMui}.
Furthermore, the Langlands- and Zelevinsky-classification are related by the Aubert-Zelevinsky involution, an involution on the set of isomorphism classes of the irreducible representations of $\gl_n(\Ff)$, see \cite{ZelII}, \cite{Aub}. It thus fell upon the classification of Bernstein and Zelevinsky in the case $G=\gl_n$ to offer itself to generalization in the modular case.
In the works of Vignéras \cite{VigInduced} and M\'inguez-S\'echerre \cite{MinSec}, this feat has been accomplished by extending their ideas. In general, the study of irreducible representations of $\gl_n$ over $\fl$ shows many similarities to the study of complex representations, but with several key differences as we will see later. So is for example no longer every cuspidal representation supercuspidal.
 At this point it has to be noted that the usefulness of standard modules in the complex case goes much further than the simple classification of irreducible representations. A prominent example, and the main motivator of this paper, are the local \emph{Rankin-Selberg L-factors} $L(s,\rho,\rho')$ of \cite{JacShaII}, which are defined for pairs of complex representations of Whittaker type $(\rho,\rho')$ and they are rational functions in $q^s$. In particular, if $G=\gl_n$, every standard module $\Sc(\pi)$ is of Whittaker type, which in turn allows to define the invariant $L(s,\Sc(\pi),\Sc(\pi'))$ for any pair of irreducible representations of $\gl_n$. It comes as no surprise that the $L$-factor $L(s,\Sc(\pi),\Sc(\pi'))$ can be explicitly computed in terms of the Langlands parameters of $\pi$ and $\pi'$. 

Passing from the complex to the modular setting, 
 the definition of local Rankin-Selberg $L$-factors has been extended in \cite{KurMat} to the more general fields $R$ we are going to consider in this paper. The authors again associate to a pair of representations of Whittaker type an $L$-factor $L(X,\pi,\pi')$, which is a rational function in $X$ and plugging in $X=q^{-s}$ in the $\CC$-case recovers the original construction of \cite{JacPiaSha}. The so obtained Rankin-Selberg L-factors satisfy a functional equation, which in turn gives rise to 
 $\epsilon$- and $\gamma$-factors associated to the pair $(\pi,\pi')$. We also recall that in \cite{KurMatII} the authors associate to each irreducible representation of $\gl_n(\Ff)$ a so-called $\Cc$-parameter $\Cc(\pi)$, which is a modular version of the Langlands parameter. Moreover, they equip $\Cc$-parameters with a tensor-product $\otimes_{ss}$ and define their $L$-, $\epsilon$- and $\gamma$- factors. 

In this paper we hope to initiate the investigation of possible candidates for standard modules in the modular setting and their properties. We do this by giving two natural possible candidates for standard modules, one obtained by reducing the ones defined over $\ql$ to $\fl$, the other one by considering the so-called \emph{generic extension map}, which has its root in the analysis of certain (quantum)-algebras associated to certain Dynkin- or affine Dynkin-quivers. We will prove that the former always contain the latter and note that at the moment the evidence points towards them being equal, although we do not treat this question in this paper.
 The above discussed properties give us furthermore natural constraints our potential standard modules should satisfy. Firstly, they should admit a modular formulation of the Langlands classification and secondly, they should allow us to define Rankin-Selberg L-functions, which should be explicitly computable using the $\Cc$-parameters of the underlying irreducible representations.
 The contents of this paper can thus be roughly divided into three parts, firstly the construction of standard modules, secondly establishing certain representation theoretic properties and finally the computation of their Rankin-Selberg $L$-factors. 

To state our results more precisely, we need to introduce some notation. 
Let $\psi\colon \Ff\ra R$ be a smooth, additive character and recall the space of Whittaker-functions 
\[\Wc(\psi)\coloneq\]\[ \{f\colon G_n\ra R:\, f\left(\begin{pmatrix}
    1&u_{1,2}&\ldots&\ldots\\
    0&\ddots&\ldots&\ldots\\
    0&\ldots&1&u_{n-1,n}\\
    0&\ldots&0&1
\end{pmatrix}g\right)=\psi(\sum_{i=1}^{n-1}u_{i,i+1})f(g),\, f\text{ locally constant}\}.\]
We also recall that a smooth representation $\pi$ of $G_n$ of finite length is called of Whittaker type if \[\dim_R\mathrm{Hom}_{G_n}(\pi,\Wc(\psi))=1,\] and in this case we denote its image by $\Wc(\pi,\psi)$. We let \[\Rep_{W,\psi,n}=\{\Wc(\pi,\psi):\pi\text{ a representation of Whittaker type of }G_n\}.\]Let us recall that to a cuspidal representation of $G_m$ we can associate $o(\rho)\in\NN\cup\{\infty\}$, with $o(\rho)=\infty$ if and only if $\mathrm{char}(R)=0$ and if $o(\rho)$ is finite, it is the minimal $k\in \ZZ_{>0}$ such that $\rho\cong\rho\abs^k$. We let moreover be $e(\rho)=\ell$ if $o(\rho)=1$ and $e(\rho)=o(\rho)$ otherwise.
To two integers $a\le b$ and $\rho$ cuspidal, we associate a segment (over R) $[a,b]_\rho=(\rho\abs^a,\ldots,\rho\abs^b)$ and we identify $[a,b]_\rho$ and $[a+o(\rho),b+o(\rho)]_\rho$. We also write $[a,b]_\rho^\lor=[-b,-a]_{\rho^\lor}$.
A finite formal sum of segments is called multisegment and we denote the set of multisegments over $R$ by $\Ms$, to which we extend the operation $(-)^\lor$. A multisegment is called aperiodic if it does not contain a sub-multisegment of the following form
\[[a,b]_\rho+[a+1,b+1]_\rho+\ldots+[a+e(\rho)-1,b+e(\rho)-1]_\rho\] and we denote the set of aperiodic multisegments by $\Ms^{ap}$. We decompose any multisegment $\fm=\fm_b+\fm_{nb},$ where $\fm_{nb}$ consists of those $[a,b]_\rho$ with $o(\rho)=1$. We let $\Msi^{ap}$ be the set of aperiodic multisegments such that $\fm=\fm_{b}$. Note that if $\mathrm{char}(R)=0$ all multisegments are aperiodic and $\Msq=\mathcal{M}ult_{\ql,\square}^{ap}$.
Then there exists a bijection, see \cite{MinSecIV} and \cite{MinSec},
\[\langle-\rangle\colon \Ms^{ap}\ra \Irr(R).\]
The definition of $\Msi^{ap}$ is motivated by the following facts.
If we denote as usual by $\times$ the normalized parabolic induction, then $\langle\fm\rangle\cong \langle\fm_b\rangle\times\langle\fm_{nb}\rangle$ and \[L(X,\Cc(\langle\fm\rangle)\otimes_{ss}\Cc(\langle\fn\rangle))=L(X,\Cc(\langle\fm_{b}\rangle)\otimes_{ss}\Cc(\langle\fn_{b}\rangle)).\]
In particular, representations with non-trivial $L$-functions are precisely those coming from $\Msi^{ap}$.
The set $\Ms^{ap}$ moreover admits an order, denoted by $\lp$, which comes from the \emph{degeneration order} on certain quiver-varieties, which we will discuss in a moment.
Our goal is then to construct a map
\[\Sc_\psi\colon \Ms^{ap}\ra \Rep_{W,\psi}=\bigcup_{n\in \NN}\Rep_{W,\psi,n}\]
satisfying, among others, the following properties.
\begin{enumerate}
    \item $L(X,\Sc_\psi(\fm)\otimes \Sc_{\psi^{-1}}(\fn))=L(X,\Cc(\langle\fm\rangle)\otimes\Cc(\langle\fn\rangle))$.
    \item $\Sc_\psi(\fm)$ admits $\langle\fm\rangle$ as a quotient and $\langle\fm\rangle$ appears in the Jordan-Hölder decomposition of $\Sc_\psi(\fm)$ with multiplicity one.
    \item $\Sc_\psi(\fm)$ admits $\langle\fm\rangle$ as its unique irreducible quotient.
\end{enumerate}
Property (2) and (3) would imply that $\dim_R\mathrm{Hom}_{G_n}(\Sc_\psi(\fm),\Sc_{\psi^{-1}}(\fm^\lor)^\lor)=1$ and every morphism in this space would factor through $\langle\fm\rangle$. The description of irreducible representation in the style of property (2) and (3) is also known as \emph{Langlands classification}, and has not yet been achieved over $\fl$, unlike in the case $\mathrm{char}(R)=0$, where the corresponding statements have first been proved in \cite{ZelII}, as was mentioned above.
In this paper we will construct such a map $\Sc_\psi$ satisfying (1) and (2), 
as well as a map \[\Sc_{gen,\psi}\colon \Msi^{ap}\ra \Rep_{W,\psi}\] satisfying properties (2) and (3). We show moreover that for $\fm\in \Msi$, $\Sc_{gen,\psi}(\fm)\subseteq\Sc_\psi(\fm)$. If $\mathrm{char}(R)=0$, these two constructions agree.

We start with a sketch of the construction of $\Sc_\psi$. If $\trho$ is a cuspidal representation over $\ql$ admitting an integral structure, we denote its reduction mod $\ell$ by $\rl(\trho)$. If $\rl(\trho)$ is again a cuspidal representation $\rho$, we say $\trho$ is a lift of $\rho$. It is then straightforward to define a lift $\tfm\in \Msq$ of $\fm\in \Msl$ by firstly saying a segment $[a',b']_\trho$ is a lift of $[a,b]_\rho$ if $\trho$ is a lift of $\rho$, $a=a'\mod o(\rho)$ and $b-a=b'-a'$, and then extending this notion linearly to all multisegments. In this case one can equip $\Sc_\psi(\tfm)$ with its natural integral structure given by those Whittaker functions which take values in $\zl$, and we denote the reduction mod $\ell$ of this specific integral structure by $\overline{\Sc_\psi(\tfm)}$.
For $\fm$ an aperiodic multisegment over $\fl$ we then denote \[\Sc_\psi(\fm)\coloneq \bigcap_{\tfm\text{ lift of }\fm}\overline{\Sc_\psi(\tfm)},\] where the intersection was taken in $\Wc(\psi)$. 
\begin{theorem}
    The map $\Sc_\psi\colon \Msl^{ap}\ra \Rep_{W,\psi}$ satisfies property (1) and if $\fm\in \mathcal{M}ult_{\fl,\square}^{ap}$, then $\langle\fm\rangle$ is a quotient of $\Sc_\psi(\fm)$ and appears with multiplicity $1$ in it.
\end{theorem} 
The proof of this theorem uses our second construction of standard modules, denoted by $\Sc_{gen,\psi}$, to which we come now.
 We will focus on the case $R=\fl$, however the construction works similarly for $\mathrm{char}(R)=0$. Moreover, for the sake of clarity we will focus in the introduction on multisegments of the form $[a_1,b_1]_\rho+\ldots+[a_k,b_k]_\rho$ for a fixed cuspidal $\rho$ with $o(\rho)$>1. We denote this set by $\Msl(\rho)$ and the aperiodic multisegments in it by $\Msl(\rho)^{ap}$. Let $Q$ be the cyclical quiver with $o(\rho)$ vertices and oriented counter-clockwise. 
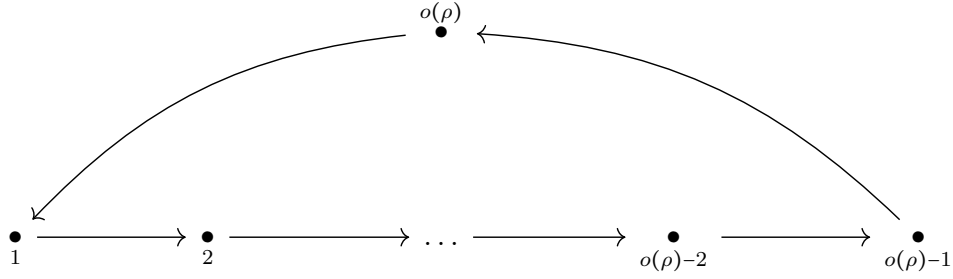
\[\begin{tikzcd}
	&&&& \overset{o(\rho)}{\bullet} \\
	\\
	\\
	 \underset{1}{\bullet}&&  \underset{2}{\bullet} & {} & \ldots &&  \underset{o(\rho)-2}{\bullet} && \underset{o(\rho)-1}{\bullet}
	\arrow[bend right=20,from=1-5, to=4-1]
	\arrow[from=4-1, to=4-3]
	\arrow[from=4-3, to=4-5]
	\arrow[from=4-5, to=4-7]
	\arrow[from=4-7, to=4-9]
	\arrow[bend right=20,from=4-9, to=1-5]
\end{tikzcd}\]
Then the isomorphism classes of finite dimensional, nilpotent $\CC$-representations of $Q$, denoted by $[Q]$,
are in bijection with $\Msl(\rho)$. The order $\lp$ on multisegments mentioned above is nothing but the degeneration order on $[Q]$ transported to $\Msl(\rho)$.
In \cite{Rin}, \cite{Bon} and \cite{Rein} the authors investigate the so called \emph{generic extension map} $*\colon [Q]\times[Q]\ra[Q]$, which sends the tuple $([M],[N])$ to $[X]$, where $X\in \mathrm{Ext}_Q^1(M,N)$ such that $\dim_\CC\mathrm{Hom}_Q(X,X)$ is minimal.
The so constructed product \[*\colon \Msl(\rho)\times \Msl(\rho)\ra \Msl(\rho)\] is associative. The subset of $\Msl(\rho)$ generated by multisegments of the form $[i,i]_\rho$ is then precisely the set of aperiodic multisegments $\Msl(\rho)^{ap}$. The following result was motivated by similar consideration of monomial basis of Quantum- and Hall algebras, see for example \cite{DenDu} or \cite{Rin}.
\begin{theorem}
    Let $\fm=[i_1,i_1]_\rho*\ldots*[i_k,i_k]_\rho\in \Msl^{ap}(\rho)$. Then
    \[\Sc_{gen,\psi}(\fm)\coloneq \Wc(\rho\abs^{i_1}\times\ldots\times\rho\abs^{i_k},\psi)\] only depends on $\fm$ and satisfies properties (1) and (2).
\end{theorem}
This construction allows us to define the desired map $\Sc_{gen,\psi}\colon \Msi^{ap}\ra \Rep_{W,\psi}$.
\begin{theorem}
    Let $\fm,\fm'\in \Msi^{ap}$. Then
    \[\Sc_{gen,\psi}(\fm*\fm')=\Wc(\Sc_{gen,\psi}(\fm)\times \Sc_{gen,\psi}(\fm'),\psi).\]
    Moreover, if $\mathrm{char}(R)=0$, $\Sc_\psi=\Sc_{gen,\psi}$ and if $R=\fl$, $\Sc_{gen,\psi}(\fm)\subseteq \Sc_\psi(\fm)$.
\end{theorem}
Although the following seems at the moment out of reach, examples of small rank seem to suggest the equality $\Sc_\psi=\Sc_{gen,\psi}$ on $\Msi^{ap}$ as well as the equality $\Sc_\psi(\fm)=\overline{\Sc_\psi(\tfm)}$ for any lift $\tfm$ of $\fm$.
\subsection*{Acknowledgments}
I would like to express my sincere gratitude to Nadir Matringe for bringing the central problem of this paper to my attention and for kindly inviting me to the NYU Shanghai in March 2024. I am also deeply grateful to Alberto Minguez for taking the time to read earlier versions of this paper and for providing valuable feedback. Their insights and support have been greatly appreciated.
This work has been supported by the projects PAT4628923 and PAT4832423 of the Austrian Science Fund (FWF).
\section{Notations}\label{S:notations}
\counterwithin{theorem}{subsection}
Let $\Ff$ be a non-archimedean local field with rings of integers $\of$ and residue field $\mathrm{k}$, whose cardinality we denote by $q$. We also choose a uniformizer $\varpi_\Ff$ of the maximal ideal of $\of$ and fix a prime $\ell$ not dividing $q$. We let $\abs$ be the absolute value of $\Ff$ such that $\lvert\varpi_\Ff\lvert=q^{-1}$.
From now on $R$ is one of the algebraically closed fields $\fl$ or $\ql$. We let $\Lambda$ be the maximal ideal of $\zl\subseteq \ql$ and fix a square root $q^{\frac{1}{2}}$ of $q$ in $\zl$. By abuse of notation, we are also going to denote its image in $\fl$ by $q^{\frac{1}{2}}$.

For $n\in \ZZ_{\ge 0}$ we let $G_n\coloneq \gl_n(F)$ and denote by $1_n$ the identity in $G_n$ and by \[w_n=\begin{pmatrix}
    &&1\\&\udots&\\1&&
\end{pmatrix}\] the anti-diagonal. Moreover, for $m,n\in \NN$, we denote by $w_{n,m}$ the block-diagonal embedding of $(1_m,w_n)$ into $G_{n+m}$ and the space $M_{n,m}$ of $n\times m$ matrices with entries in $\Ff$.

Inside $G_n$ we consider for $k\in \ZZ$ the closed subsets
\[G_n^k\coloneq\{g\in G_n:\, \mathrm{val}_\Ff(\det(g))=k\},\]
where $\mathrm{val}_\Ff$ denotes the valuation of $\Ff$ with $\mathrm{val}_\Ff(\varpi)=1$.
We let $B_n$ be the Borel subgroup of $G_n$ of upper diagonal matrices and $N_n$ its unipotent subgroup. 
More generally, if $\alpha=(\alpha_1,\ldots,\alpha_k)$ is a partition of $n$, we denote by $P_\alpha$ the parabolic subgroup containing $B_n$ with Levi-component $G_\alpha\coloneq G_{\alpha_1}\times\ldots\times G_{\alpha_k}$. Its opposite parabolic subgroup, $\overline{P_\alpha}$ is conjugated to $P_{(\alpha_k,\ldots,\alpha_1)}$.
Moreover, we denote by
\[ H_{n,m}\coloneq\left\{\begin{pmatrix}
    g&x\\0&n
\end{pmatrix}:\, g\in G_n,\,x\in M_{n,m},\, n\in N_m\right\}\subseteq G_{n+m}\] and let $P_n\coloneq H_{n-1,1}$ be the mirabolic subgroup of $G_n$.
Let $C_c^\infty(\Ff^n)$ be the space of $R$-valued locally constant and compactly supported functions on $\Ff^n$ and we set $\epsilon_n\coloneq(0,\ldots,0,1)\in \Ff^n$.

\subsection{ Representations}
In the next sections we recall the usual setup for $\ell$-adic and $\ell$-modular representation theory, \emph{c.f.} \cite{VigBook}.
Let $G\subseteq G_n$ be a closed subgroup. We denote by $\Rep(G,R)$ the category of smooth, admissible representations of finite length of $G$. Whenever possible we will omit the field of coefficients $R$. Let $\alpha=(\alpha_1,\ldots,\alpha_k)$ a partition of $n$.
We let $\Rep_\alpha=\Rep(G_\alpha)$, $\Irr_\alpha$ be the set of isomorphism classes of irreducible representations in $\Rep_\alpha$ and set \[\Rep\coloneq\bigcup_{n\in \NN}\Rep_n,\, \Irr\coloneq\bigcup_{n\in \NN}\Irr_n.\]
If $G\subseteq H\subseteq G_n$ are closed subgroups, we denote the functors of normalized induction by $\ind_H^G$ and its compactly supported version by $\Ind_H^G$. 
We recall the normalized Jacquet functor and parabolic induction corresponding to parabolic subgroups $P=MN$ of $G_\alpha$, which give rise to the exact functors
\[r_P \colon \Rep_\alpha\ra \Rep(M),\, \Ind_{P}^{G_\alpha}\colon \Rep(M)\ra \Rep_\alpha.\]
We write $r_\alpha\coloneq r_{P_\alpha}$ and $\ol{r_\alpha}\coloneq r_{\ol{P_\alpha}}$. 
Recall that by Frobenius reciprocity respectively Bernstein reciprocity $r_{\alpha}$ respectively $\ol{r_\alpha}$ is the left adjoint  respectively right adjoint of $\Ind_{P_\alpha}^{G_n}$. By abuse of notation we will also notate the maps they induce between the respective Grothendieck groups by the same letters.
As is convention, we will write
\[\pi_1\times\ldots\times\pi_k\coloneq \Ind_{P_\alpha}^{G_n}(\pi_1\otimes\ldots\otimes\pi_k).\]
If $\chi$ is a smooth character of $\Ff$ and $\pi\in \Rep$, we denote by $\chi\pi=\chi(\det)\pi$. If $\pi\in \Rep_n$, we denote the corresponding element in the Grothendieck of $\Rep_n$ by $[\pi]$ and denote by the length of $\pi$ the number of its irreducible subquotients counted with multiplicity. If $\pi$ is a representation of $G_n$, we denote by $\pi^\mathfrak{c}$ the representation obtained by twisting $\pi$ by $g\mapsto w_n(g^{-1})^tw_n$. We recall that if $\pi$ is irreducible and $\ell>2$, then $\pi^\mathfrak{c}\cong \pi^\lor$, see for example \cite{VigBook} or \cite[Remark 2.7]{MinSec}, and \[{(\pi_1\times\pi_2)}^\mathfrak{c}={\pi_2}^\mathfrak{c}\times {\pi_1}^\mathfrak{c}.\] Finally, we recall that parabolic induction on $G_n$ is commutative on the level of Grothendieck groups, in particular for $\pi,\pi'\in \Irr$ such that $\pi\times\pi'$ is irreducible, $\pi\times\pi'\cong\pi'\times \pi$, see \cite[1.16]{VigBook} for $\ell>2$, \cite[Theorem 1.9]{ZelII} for the case $R=\mathbb{C}$ and \cite[Proposition 2.6]{MinSec} in general.
\subsubsection{ Cuspidal representations}
A representation $\rho\in\Irr_n$ is called cuspidal if for all non-trivial partitions $\alpha$ of $n$, $r_\alpha(\rho)=0$. It is called supercuspidal if there exists a non-trivial partition $\alpha$ and $\pi\in \Irr_\alpha$ such that $\rho$ appears as a subquotient of $\Ind_{P_\alpha}^{G_n}\pi$. In general, supercuspidal implies cuspidal and if $R=\ql$, cuspidal implies supercuspidal. We denote the subset of $\Irr_n$ consisting of cuspidal respectively supercuspidal representations by $\cusp_n$ respectively $\Scu_n$ and define \[\cusp\coloneq\bigcup_{n\in \NN}\cusp_n,\, \Scu\coloneq\bigcup_{n\in \NN}\Scu_n.\] 
We also recall the notation of cuspidal support, \emph{i.e.} if $\pi\in \Irr$ there exist up to possible permutation and isomorphism unique $\rho_1,\ldots,\rho_k\in \cusp$ such that $\pi\hra \rho_1\times\ldots\times\rho_k$. We denote then by $\cus(\pi)\coloneq [\rho_1]+\ldots+[\rho_k]$ the cuspidal support of $\pi$. Weakening in the above definition the condition of being a subrepresentation to being a subquotient and cuspidal representations to supercuspidal representations, gives rise to the supercuspidal support $\scus(\pi)$.
If at any point the field of coefficients will be important, we will add in brackets to the respective category or set, e.g. $\Irr(\fl)$ versus $\Irr(\ql)$. 

We recall the following notions, see for example \cite[§3.4, §4.5]{MinSecSte}.
Let $\rho\in \cusp$ and recall that $\rho\times\chi\rho$ is reducible if and only if $\chi\cong \abs^{\pm}$.
We recall the cuspidal line \[\ZZ[\rho]\coloneq \{[\rho \abs^k]:\, k\in \ZZ\}\] and denote the cardinality of $\ZZ[\rho]$ by $o(\rho)$. Note that $o(\rho)$ is finite if and only if $R=\fl$. Set \[e(\rho)\coloneq \begin{cases}
    o(\rho)&\text{ if }o(\rho)>1,\\
    \ell&\text{ otherwise.}
\end{cases}\]
Moreover, one can associate to $\rho$ an integer $f(\rho)$ via type theory.
If $o(\rho)>1$, $o(\rho)$ is the order of $q^{f(\rho)}$ in $R$.
We let $\scu_n$ be the subset of $\cusp_n$ consisting of $\rho$ such that $o(\rho)>1$, \emph{i.e.} those $\rho$ which satisfy that $\rho\times\rho$ is irreducible.
Finally, we say $\rho$ and $\rho'$ are in different cuspidal lines if $[\rho']\notin\ZZ[\rho]$
and set $\scu\coloneq \bigcup_{n\in \NN}\scu_n$.
\subsection{ Multisegments}
We now recall the combinatorics of multisegment, \emph{cf. }\cite{ZelII}, \cite{MinSec}.
Let $\rho\in \cusp_m$ and $a\le b\in \ZZ$. A segment is a sequence
\[[a,b]_\rho=([\rho\abs^a],\ldots,[\rho\abs^b])\] and two segments $[a,b]_\rho$ and $[a',b']_{\rho'}$ are equal if and only if $\rho'\abs^a\cong \rho\abs^a$ and $b-a=b'-a'$. We also let $[a,b]_\rho^\lor=[-b,-a]_{\rho^\lor}$.
Define \[{}^-[a,b]_\rho=[a+1,b]_\rho,\,[a,b]_\rho^{-}=[a,b-1]_\rho,\,{}^+[a,b]_\rho=[a-1,b]_\rho,\,[a,b]_\rho^{+}=[a,b+1]_\rho,\]\[ a_\rho([a,b]_\rho)=a\in \begin{cases}
    \ZZ/(o(\rho)\ZZ)&o(\rho)<\infty,\\\ZZ&o(\rho)=\infty
\end{cases},\, b_\rho([a,b]_\rho)=b\in \begin{cases}
    \ZZ/(o(\rho)\ZZ)&o(\rho)<\infty,\\\ZZ&o(\rho)=\infty
\end{cases}.\]
We define length and degree of a segment as $l([a,b]_\rho)\coloneq b-a+1\in \ZZ$ and $\deg([a,b]_\rho)=(b-a+1)m\in \ZZ$. The cuspidal support of $[a,b]_\rho$ is defined as $\mathrm{cusp}([a,b]_\rho)\coloneq[\rho\abs^a]+\ldots+[\rho\abs]^b$. By abuse of notation we will often write $\rho$ for the segment $[0,0]_\rho$.

A multisegment is a formal finite sum and we extend the length $l$, the degree $\deg$, the notion of cuspidal support, and $(-)^\lor$ linearly. We let $\Ms$ be the set of multisegments and $\Ms(\rho)$ be the set of multisegments consisting only of segments of the form $[a,b]_\rho$. A multisegment is called aperiodic if it does not contain a sub-multisegment of the form
\[[a,b]_\rho+[a+1,b+1]_\rho+\ldots+[a+e(\rho)-1,b+e(\rho)-1]_\rho.\]
For any subset $\mathcal{N}\subseteq \Ms$, we denote by $\mathcal{N}^{ap}$ the aperiodic multisegments in $\mathcal{N}.$ Moreover, for any $\fm\in \Ms$, we decompose $\fm=\fm_b+\fm_{nb}$ with $\fm_b$ consisting of all segments $[a,b]_\rho$ in $\fm$ with $\rho\in \scu$. We let $\Msi$ be the set of multisegments such that $\fm=\fm_b$.
\begin{theorem}[{\cite[§9]{MinSec}, \cite[§6]{ZelII}}]\label{T:msZ}
There exists a surjective map
\[\Z\colon \Ms\ra \Irr\]
satisfying the following.
\begin{enumerate}
    \item $\Z(\fm)\in \Irr_{\deg(\fm)}$.
    \item $\Z$ restricted to aperiodic multisegments is a bijection.
    \item $\Z(\fm)^\lor\cong \Z(\fm^\lor)$.
    \item For $\fm\in \Ms^{ap}$, $\cus(\Z(\fm))=\cus(\fm)$. If $\fm\in\Ms(\rho)^{ap}$ with $\rho\in\Scu$, then $\scus(\Z(\fm))=\cus(\Z(\fm))$. 
    \item If $\fm=\fm_1+\fm_2$ then $\Z(\fm)$ is a subquotient of $\Z(\fm_1)\times \Z(\fm_2)$ and appears with multiplicity one in its Jordan Hölder decomposition.
    \item $\Z([a,b]_\rho)$ is a subrepresentation of $\rho\abs^a\times\ldots \rho\abs^b$.
    \item If $\rho_1,\ldots,\rho_l\in \cusp$ lie in pairwise different cuspidal lines and $\fm=\fm_1+\ldots+\fm_k$ with $\fm_i\in \Ms(\rho_i)$,
    \[\Z(\fm)\cong \Z(\fm_1)\times\ldots\times \Z(\fm_k).\]
\end{enumerate}
\end{theorem}
If $\fm$ consists only of segments of length $1$, we call $\Z(\fm)$ non-degenerate.
Finally, we recall the Aubert-Zelevinsky involution, \emph{c.f.} \cite{Aub}, \cite{ZelII}, \cite{MinSecIV}, $(-)^*\colon \Irr\ra \Irr$ and denote \[\langle-\rangle\coloneq (\Z)^*\colon \Ms^{ap}\ra\Irr.\]
This map preserves the cuspidal support
and if $\fm,\fm_1,\ldots,\fm_k\in \Ms^{ap}$, the multiplicity of $\langle\fm\rangle$ in $\langle\fm_1\rangle\times\ldots\times\langle\fm_k\rangle$ is the multiplicity of $\Z(\fm)$ in $\Z(\fm_1)\times\ldots\Z(\fm_k)$.
\begin{lemma}\label{L:irred}
Let $\rho,\rho'\in \scu$ be in different cuspidal lines. Moreover, let $\pi$ be an irreducible subquotient of $\rho_1\times\ldots\times\rho_k,\,\rho_i\in\ZZ[\rho]$ and $\pi'$ be an irreducible subquotient of $\rho_1'\times\ldots\times\rho_l',\,\rho_i'\in\ZZ[\rho]$. Then $\pi\times \pi'$ is irreducible
\end{lemma}
\begin{proof}
    By the classification of cuspidal representations in \cite[§6]{MinII}, the cuspidal lines of the cuspidal supports of $\pi$ and $\pi'$ do not intersect. Indeed, if $\sigma$ is a cuspidal representation appearing in the cuspidal support of $\pi$, its supercuspidal support consists of representations in $\ZZ[\rho]$, see \cite[Theorem 9.36(3)]{MinSec}. By the uniqueness of the cuspidal and supercuspidal support, it follows that the cuspidal representations appearing in $\pi$ and $\pi'$ lie in different cuspidal lines.
    The claim follows from the property (7) in \Cref{T:msZ}.
\end{proof}
Let $\fm\in \Ms$ and $\De=[a,b]_\rho,\, \De'=[a',b']_\rho$ be two segments in $\fm$ with \[a+1\le a'\le b+1\le b'.\] In this case we call $\De$ and $\De'$ linked. An \emph{elementary operation} on $\fm$ refers to changing the segment in the following way: 
\[\fm\mapsto \fm-\De-\De'+[a,b']_\rho+[a',b]_\rho.\]
We will say that the elementary operation is of the form $[a,b]_\rho+[a',b']_\rho\mapsto [a,b']_\rho+[a',b]_\rho$.
We write $\fn\lp \fm$ if $\fn$ can be obtained by repeated applications of elementary operations to $\fm$ and $\fn\lpp\fm$ if $\fn\lp\fm$ and $\fm\neq \fn$. The multisegment $\fm=\De_1+\ldots+\De_k$ is called unlinked, if for all $\ain{i\neq j}{1}{k}$ $\De_i$ and $\De_j$ are unlinked. 
\begin{lemma}[{\cite[Theorem 7.1]{ZelII}, \cite[Theorem 6.4.1]{DroMod}}]\label{L:ordersubsubquot}
    Let $\fm=\De_1+\ldots+\De_k\in \Ms$ and $\fn\in \Ms^{ap}$. Then $\langle\De_1\rangle\times\ldots\times\langle\De_k\rangle$ contains $\langle\fn\rangle$ as an irreducible subquotient if and only if $\fn\lp\fm$.
    In particular, $\langle\fm\rangle\cong \langle\De_1\rangle\times\ldots\times\langle\De_k\rangle$ if and only if $\fm$ is unlinked.
\end{lemma}
\begin{proof}
    The case in $R=\ql$ was treated in \cite[Theorem 7.1]{ZelII} and the case $R=\fl$ in \cite[Theorem 6.4.1]{DroMod}. Note that in \cite{DroMod}, the claim was proven for $\Z(-)$ instead of $\langle-\rangle$, however applying the Aubert-Zelevinsky yields the equivalence. 
    
    To see the claim regarding the irreducibilty, note that the induced representation is irreducible if and only if it contains no irreducible subquotient $\langle\fn'\rangle$ for $\fm\neq \fn'\in \Ms$. Note moreover, that if $\fn'$ is not aperiodic, it is easy to construct $\fn''\lpp \fn'$ via one elementary operation, and hence $\fn''\lpp\fn'\lp\fm$. Repeating this process, we arrive at $\fn\lpp\fn'\lp\fm$ with $\fn$ aperiodic, and hence by the above lemma $\langle\fn\rangle$ is a subquotient of the induced representation.
    Thus the induced representation is irreducible if and only if it contains no $\langle \fn\rangle$ with $\fn\in \Ms^{ap}$, which by the above can only happen if and only if $\fn\lp\fm$. 
\end{proof}
Assume now for a moment that $R=\ql$ and let $\De,\De'$ be two segments. We say $\De$ preceds $\De'$ if $\De=[a,b]_\rho,\, \De'=[a',b']_\rho$ and \[a+1\le a'\le b+1\le b'.\]
Let $\fm=\De_1+\ldots+\De_k\in \Msq$. We say $(\De_1,\ldots,\De_k)$ is in \emph{arranged form} if for all $\ain{i,j}{1}{k},\, i<j$ $\De_i$ does not precede $\De_j$. Any multisegment over $\ql$ admits an arranged form and any two arranged forms can be obtained from each other by repeatedly changing the order of two neighboring unlinked segments, \emph{i.e.} replacing
$(\ldots,\De_i,\De_{i+1},\ldots)\mapsto (\ldots,\De_{i+1},\De_i,\ldots)$ if $\De_i$ and $\De_{i+1}$ are unlinked.
\subsection{ Integral structures}\label{S:intstruc}
We recall the following results of \cite[I.9]{VigBook} on integral structures.
Let $\alpha$ be a partition of $n\in \NN$ and $(\widetilde{\pi},V)\in \Rep_\alpha(\ql)$. We recall that an integral structure of $\widetilde{\pi}$ is a $\zl$-lattice $\mathfrak{l}$ in $V$ which generates $V$, \emph{i.e.} the natural map induces an isomorphism $\mathfrak{l}\otimes_\zl\ql\cong V$. If $\widetilde{\pi}$ admits an integral structure it is called \emph{integral}.
For $\widetilde{\pi}\in \Rep_\alpha$ integral and $\mathfrak{l}$ an integral structure, $\mathfrak{l}\otimes_\zl\fl$ is an object in $\Rep_\alpha(\fl)$. We denote by
\[\rl(\widetilde{\pi})\coloneq [\mathfrak{l}\otimes_\zl\fl],\] which is an element in the Grothendieck group of $\Rep_\alpha(\fl)$. It is $\rl(\widetilde{\pi})$ which is independent of the chosen integral structure; the representation $\mathfrak{l}\otimes_\zl\fl$ is highly dependent on the specific $\mathfrak{l}$ and in general not irreducible.
Moreover, the representation $\Ind_{P_\alpha}^{G_n}\widetilde{\pi}$ is integral, with the natural integral structure $\Ind_{P_\alpha}^{G_n}\mathfrak{l}$, \emph{i.e.} the $\mathfrak{l}$-valued functions.
Then
\[(\Ind_{P_\alpha}^{G_n}\mathfrak{l})\otimes_\zl\fl\cong \Ind_{P_\alpha}^{G_n}(\mathfrak{l}\otimes_\zl\fl).\]
Moreover, if $\widetilde{\tau}$ is a subrepresentation respectively quotient of $\widetilde{\pi}$, $\mathfrak{l}\cap \widetilde{\tau}$ respectively the image of $\mathfrak{l}$ in $\widetilde{\tau}$ are integral structures of $\widetilde{\tau}$.

Finally, the functor $r_\alpha$ and $\rl$ commute on the level of Grothendieck groups by \cite[Proposition 1.4(i)]{Dat}, \emph{i.e.}
\[r_\alpha(\rl(\widetilde{\pi}))=\rl(r_\alpha(\widetilde{\pi})).\]

An integral representation $\widetilde{\pi}\in \Irr(\ql)$ is called a \emph{lift} of $\pi\in \Irr(\fl)$ if $\rl(\widetilde{\pi})=[\pi]$.
If $\De=[a,b]_\rho$ is a segment over $\fl$, a segment $\TD=[a',b']_\trho$ over $\ql$ is called a lift of $\De$ if $\trho$ is a lift of $\rho$, $b'-a'=b-a$ and $a=a'\mod o(\rho)$. Finally, we extend this notation of a lift to segments $\Msl$ linearly.

In \cite[III]{VigBook} the Bushnell–Kutzko construction for $\ell$-adic and modular representations was carried out and in particular, it was shown that every cuspidal representation over $\fl$ admits a cuspidal lift $\trho$ to $\ql$. Moreover, if $\trho$ is such a lift of $\rho$, $f(\rho)=f(\trho)$.
\begin{lemma}[{\cite[Theorem 9.39]{MinSec}}]\label{T:multone}
    Let $\fm\in \Msl^{ap}$ and $\tfm$ a lift of $\fm$. Then
    $\Z(\fm)$ appears with multiplicity $1$ in $\rl(\Z(\tfm))$ and $\langle\fm\rangle$ appears with multiplicity $1$ in $\rl(\langle \tfm\rangle)$.
\end{lemma}
\begin{proof}
    The claim for $\Z$ can be found in \cite[Theorem 9.39]{MinSec}, the one for $\langle-\rangle$ follows easily from the definition of $\langle-\rangle$ in \cite{Aub}. Indeed let us recall the definition of $(-)^*$. For $[\pi]$ an element in the Grothendieck group of $\Rep_n$, one first defines a non-zero element $\mathrm{D}([\pi])$ in the same Grothendieck group, which is represented by the cohomology of a complex obtained from $[\pi]$ by applying repeatedly the Jacquet functor and parabolic induction. In particular, if $[\pi]\le [\tau]$, then $\mathrm{D}([\pi])\le \mathrm{D}([\tau])$. Moreover, if $\pi$ is irreducible, there exists a unique irreducible summand $[\pi^*]$
    in $\mathrm{D}(\pi)$ with the same cuspidal support as $\pi$ and if $R=\ql$, then $\mathrm{D}(\pi)=[\pi^*]$.
The above construction implies that $\mathrm{D}$ commutes with $\rl$.
    In particular, $\rl(\langle\tfm\rangle)=\rl(\mathrm{D}(\Z(\tfm)))=\mathrm{D}(\rl(\Z(\tfm)))$, where the latter contains $\mathrm{D}(\Z(\fm))\ge [\langle\fm\rangle]$ by the first claim.
\end{proof}
\subsection{ Generic extensions}\label{S:genext}
In this section we recall the composition algebra of the cyclic quiver, \emph{c.f.} \cite{Rin}.
We fix $1<n\in \NN\cup\{\infty\}$ and consider the cycle quiver $Q$ with $n$ vertices, denoted by $I=\ZZ/ (n\ZZ)$, and oriented counter-clockwise. If $n=\infty$, we let $Q$ be the quiver $A_\infty$, oriented by pointing all arrows to the right.
We recall that a representation of $Q$ is nothing but a finite-dimensional $\ZZ/ (n\ZZ)$-graded $\CC$-vector $V$ space together with a linear map $T\colon V\mapsto V$ of weight $1$. We call the representation nilpotent if for large enough $N\in\NN$, $T^N=0$ and associate to $V$ the dimension vector $\grdim V\in \ZZ^n$, where the $i$-th entry of $\grdim V$ equals to $\dim_\CC V_i$.
Given a dimension vector $\db$ we let $E_\db=\{(V_\db,T)\}$ be the set of finite dimensional, nilpotent representations of $Q$ over $\CC$ with underlying graded vector-space \[V_\db\coloneq \bigoplus_{i\in \mathbb{Z}/(n\ZZ)}\CC^{\db_i}.\] We let $G_\db\coloneq \gl_{d_1}(\CC)\times\ldots\gl_{d_n}(\CC)$, which acts on $E_\db$ by conjugation. The orbits of this action are naturally parametrized as follows.
We let $\MsQ$ be the set of multisegments, \emph{i.e.} the formal finite sums of segments $[a,b],\,a\le b\in \ZZ$ up to the equivalence $[a,b]\sim[a+n,b+n]$. We write $[a,b]^\lor=[-b,-a]$, $l([a,b])=b-a+1$ and extend these operations linearly to $\MsQ$.
To $[a,b]\in \MsQ$ we associate the indecomposable representation $\lambda([a,b])$, whose underlying vector space has as a basis $b-a+1$ vectors $e_1,\ldots,e_{b-a+1}$, with $e_i$ in degree $i\mod n$ and $T(e_i)=e_{i+1}$ for $i\le b-a$ and $T(e_{b-a+1})=0$. The dimension vector of this representation, denoted by $\grdim [a,b]$, has as its $i$-th entry the number of integers $x$ such that $a\le x\le b$ and $x=i\mod n$.
More generally, we associate to a multisegment $\fm=[a_1,b_1]+\ldots+[a_k,b_k]$
the representation $\lambda(\fm)\coloneq \lambda([a_1,b_1])\oplus\ldots\oplus \lambda([a_k,b_k])$. We denote its dimension vector $\grdim\fm=\grdim[a_1,b_1]+\ldots+\grdim[a_k,b_k]$.
We call a multisegment aperiodic if it does not contain a multisegment of the form $[a,b]+\ldots+[a+n-1,b+n+1]$ and we denote the set of aperiodic multisegments by $\MsQ^{ap}$.

The $G_\db$-orbits $[E_\db]$ of $E_\db$ are then via the above construction in bijection with multisegments $\fm$ such that $\db=\grdim\fm$. For $M$ a representation in $E_\db$, we denote by $[M]=G_\db\cdot M$ its orbit.

We write $[M]\lp [N]$ for two orbits in $E_\db$ if $[N]$ is in the closure of $[M]$ with respect to the analytic topology. This relation
gives rise to the so-called \emph{degeneration order}.
We also recall the associative product
\[*\colon[E_\db]\times [E_{\db'}]\ra [E_{\db+\db'}]\]
given by sending $([M],[N])\mapsto [M]*[N]$ to the orbit of their \emph{generic extension},
\emph{c.f.} \cite[§3]{DenDu} for cyclic quivers and \cite[§2]{Rein} for Dynkin quivers. The generic extension of two representations $M$ and $N$ is defined as the set of
$X\in \mathrm{Ext}_Q^1(M,N)$ for which the dimension of the complex algebraic variety $\overline{[X]}$ is maximal, or equivalently, for which $\dim_\C\Ho_Q(X,X)$ is minimal. Any $X,X'$ in the generic extension of $M$ and $N$ lie in the same equivalence class, thus we can define $[M]*[N]$ as $[X]$ for some $X$ in the generic extension of $M$ and $N$.
\begin{lemma}[{\cite[Proposition 2.4]{Rein}, see also \cite[Proposition 3.4]{DenDu}}]\label{L:genorder}
Let $M\in E_\db, N\in E_{\db'}, X\in E_{\db+\db'}$. Then $[M]*[N]\lp [X]$ if and only if there exist $[M]\lp [M']$, $[N]\lp [N']$ such that there exists a short exact sequence $0\ra M'\ra X\ra N'\ra 0$. In particular, if $[M]\lp [M']$, $[N]\lp [N']$ then $[M]*[N]\lp [M']*[N']$.
\end{lemma}
We write for a word $w=i_1\ldots i_k$ of indices in $I$
\[\fm_{gen}(w)=[i_1,i_1]*\ldots*[i_k,i_k]\] and denote the set of words in $I$ by $\Omega$.
We can describe $\fm_{gen}(w)$ recursively as follows, see for example \cite[p. 285 and Proposition 3.7]{DenDu}. 

For $\fm\in\MsQ$ and $i\in I$ we let \[i+\fm\coloneq\begin{cases}
    \fm+[i,i]&\text{ if there does not exist a segment of the form }[i+1,b]\text{ in }\fm,\\
    \fm+{}^+\De-\De&\text{ where }\De\text{ is the longest segment of the form }[i+1,b]\text{ in }\fm.
\end{cases}\]
Similarly, we let \[\fm+i\coloneq\begin{cases}
    \fm+[i,i]&\text{ if there does not exist a segment of the form }[a,i-1]\text{ in }\fm,\\
    \fm+\De{}^+-\De&\text{ where }\De\text{ is the longest segment of the form }[a,i-1]\text{ in }\fm.
\end{cases}\]
\begin{lemma}[{\cite[Proposition 3.7]{DenDu}}]\label{L:rules}
Using the above notation, we have $\fm_{gen}(iw)=i+\fm_{gen}(w)$ and $\fm_{gen}(wi)=\fm_{gen}(w)+i$. 
\end{lemma}
From now on we implicitly identify the isomorphism classes of representations of $Q$ with $\MsQ^{ap}$.
\begin{theorem}\label{T:serre}
    The map $\fm_{gen}\colon\Omega\ra\MsQ$ has image $\MsQ^{ap}$ and two words $w$ and $w'$ give rise to the same multisegment if and only if they are related by the degenerate Serre relations.
    \begin{enumerate}
        \item $ij=ji$ if $i$ and $j$ are not neighbours.
        \item $i(i+1)i=ii(i+1)$ and $i(i+1)(i+1)=(i+1)i(i+1)$ if $n>2$.
        \item $i(i+1)ii=ii(i+1)i$ if $n=2$.
    \end{enumerate}
\end{theorem}
\begin{proof}
For a proof see \cite[Theorem A, Theorem B]{Rin}, where the author considers the algebras $H_t(Q)$ we consider below at the specialisation $t$ being a prime number. We give a sketch of the argument for the sake of the reader, see also \cite[Theorem 4.2]{Rein}.
Let $t$ be an indeterminante and consider the $\mathbb{Q}[t]$-algebra $H_t(Q)$, which is generated by variables $x_1,\ldots,x_n$ satisfying the relations
\begin{enumerate}
    \item $x_ix_j=x_jx_i$ if $i,j$ are not neighbors in $Q$.
    \item $x_i^2x_{i+1}-tx_{i+1}x_i^2=(t+1)x_ix_{i+1}x_i,\,x_ix_{i+1}^2-tx_{i+1}^2x_i=(t+1)x_{i+1}x_ix_{i+1}$ if $n>2$.
    \item \[tx_2x_1^3-(t^2+t+1)(x_1x_2x_1^2+x_1^2x_2x_1)=tx_1^3,\, tx_1x_2^3-(t^2+t+1)(x_2x_1x_2^2+x_2^2x_1x_2)=tx_2^3\]
    if $n=2$.
\end{enumerate}
Given a dimension vector $\db$, we can ask for the rank of the free $\mathbb{Q}[t]$-submodule $H_t^\db(Q)$ of $H_t(Q)$ spanned by the monomials containing $x_i$ with multiplicity $\db_i$. Specializing at $t=1$, we have that $H_1(Q)$ is the universal enveloping algebra of a certain upper-triangular part of a Lie algebra, and obtain by \cite[Proposition 7.2]{Rin} that the rank equals to the number of aperiodic elements in $\MsQ$ with dimension-vector $\db$. The author achieves this by constructing an explicit PBW -basis of the latter space.
But on the other hand, we obtain by specializing at $t=0$ the associative algebra on $\Omega$ subject to the Serre relations. We have a morphism of algebras
\[H_0(Q)\ra \mathbb{Q}[\MsQ^{ap}],\]
where the right side is equipped with the generic extension product and we send $i\mapsto [i,i]$. Indeed, one just needs to check that the expressions involved in the Serre relations give rise to the same multisegments. For example, if $n>2$, \[i(i+1)i=[i,i+1]+[i,i]=ii(i+1).\] The other relations can be checked similarly.
The above map is surjective by \cite[§4]{Rin}, and since for $\db$ a cuspidal support, $H_0^\db(Q)$ has image in the multisegments with cuspidal support $\db$, the claim follows, because the dimensions agree.
\end{proof}
\begin{rem}
    There exists a second map $\fm_{crys}\colon \Omega\ra \MsQ$ linked to $\rho$-derivatives and certain crystal bases of quantum groups, see for example \cite{DroMod} or \cite{LecVas}. Even though $\fm_{crys}$ admits a similar, although slightly more involved, recursive description as $\fm_{gen}$, these two maps differ in general.
\end{rem}
If $w=i_1\ldots i_k\in \Omega$, we let $w^\lor=(-i_k)\ldots(-i_1)$. It follows from the definitions that $(i+\fm)^\lor=\fm^\lor+(-i)$ and hence \[\fm_{gen}(w^\lor)=\fm_{gen}(w)^\lor.\]
\subsubsection{ Degeneration order and Serre relations}\label{S:order}

In this subsection we describe how, given two words $v,w\in\Omega$, one can decide whether $[\fm_{gen}(v)]\lp[\fm_{gen}(w)]$. 
Elementary operations on $\MsQ$ can be defined in a manner completely analogously to those on $\Ms$ and the first step in answering the above question is the following.
\begin{prop}[{\cite[Theorem 2.2]{ZelIII}, \cite{orbitclos}, \cite[Theorem 3.12]{orbitclos2}}]\label{P:oclos}
The degeneration order and the order by elementary operation are equivalent, \emph{i.e.} \[[\lambda(\fm)]\lp [\lambda(\fn)]\] if and only if $\fm$ can be obtained by $\fn$ via finitely many elementary operations.
\end{prop}
Let $v,w\in \Omega$. We write $v\le w$ if $v$ can be obtained from $w$ by applying a finite sequence of the following moves.
\begin{enumerate}
    \item $ij\mapsto ji$ if $i$ and $j$ are not neighbors in $I$.
    \item If $n>2$, $(i+1)i\mapsto i(i+1)$ and $ii(i+1)\mapsto i(i+1)i$.
    \item If $n=2$, $i(i+1)(i+1)\mapsto (i+1)i(i+1)$.
\end{enumerate}
Note that in particular the Serre relations are invariant under the above moves.
The goal of this section is to prove the following proposition.
\begin{prop}\label{P:ordercomp}
    Let $v,w\in \Omega$. If $v\ge w$ then $\fm_{gen}(v)\le\fm_{gen}(w)$. If $n$ is infinite, then $\fm_{gen}(v)\le\fm_{gen}(w)$ implies $v\ge w$.
\end{prop}
It is possible that there exists a more straightforward proof than the direct, purely combinatorial proof we offer, namely by relating the questions to properties of PBW-basis as in the proof of \Cref{T:serre}.
Note that the second implication in the above proposition is wrong if $n$ is not infinite. For example if $n=3$, we can construct the following example.
$v=002211102$ and $w=022110021$. Then $w$ is maximal with respect to the above order but
\[\fm_{gen}(v)=[0,2]+[0,1]+[2,3]+[2,2]+[1,1]>\fm_{gen}(w)=[0,2]+[-1,1]+[2,3]+[1,1].\]
\begin{proof}[Proof of \Cref{P:ordercomp}]
    Assume first that $v\le w$. Then it suffices by \Cref{L:genorder} to show that $[i,i]+[i+1,i+1]=(i+1)+i\gp i+(i+1)=[i,i+1]$ and $i+i+(i+1)= i+(i+1)+i$ if $n>2$. The first is seen to be a simple elementary operation whereas the second is a Serre relation. The case $n=2$ can be checked analogously.

    For the other direction (in the case $n$ being infinite) write $\fn=\fm_{gen}(v)$ and $\fm=\fm_{gen}(w)$, $\fn\lpp\fm$. Using the Serre-relations, it suffices to find some words $v'\lp w'$ with $\fn=\fm_{gen}(v')$ and $\fm=\fm_{gen}(w')$. 
    We thus will construct for any aperiodic $\fn\lpp\fm$ two such words $v'\le w'$ via induction on the length of $\fm$. 
    
     It clearly suffices to treat the case where $\fn$ is obtained from $\fm$ via one elementary operation $[a,b]+[a',b']\mapsto[a,b']+[a,b']$, $a\le a'+1\le b\le b'+1$ and $\fn\lpp\fm$ is minimal, \emph{i.e.} there exists no $\fk$ with $\fn\lpp\fk\lpp\fm$. This implies in particular the following. If there exists $[c,d]\in \fm$ with $a\le c\le a'\le b\le d\le b',$ then $[c,d]$ is one of $[a,b],[a',b'],[a,b']$ or $[a',b]$. Indeed, if there would exist a $[c,d]$ different from these four segments we could decompose the above elementary operation in the following way. Assume that $a<c<a'$, the other case follows analogously. Then we can first apply $[a',b']+[c,d]\mapsto [a',d]+[c,b']$ to $\fm$, then $[a,b]+[a',d]\mapsto [a',b]+[a,d]$ and finally $[a,d]+[c,b']\mapsto [a,b']+[c,d]$ to obtain $\fn$, a contradiction to the minimality of the elementary operation. 

    Note that it makes no difference whether we prove the claim for $\fm$ or $\fm^\lor$, thus we can assume without loss of generality that $b-a\ge b'-a'$. We now let $[c,d]$ be the longest segment in $\fm$ with $a\le c$ such that $[c+1,d+1]$ does not appear in $\fm$. Set $\fm'=\fm-[c,d]+[c+1,d]$ and by construction $c+\fm'=\fm$.
    If $[c,d]$ is not the precise copy of $[a',b']$ involved in the elementary operation, it follows straightforwardly that the elementary operation descends to an elementary operation on $\fm'$ yielding a multisegment $\fn'\lp\fm'$ with $c+\fn'=\fn$. In this case the claim follows by the induction hypothesis. On the other hand, if $[c,d]=[a',b']$, it follows by construction and the assumption $b-a\ge b'-a'$ that $b-a=b'-a'$ and for all $\ain{i}{0}{a'-a}$, $[a+i,b+i]$ appears in $\fm$. By the minimality of the elementary operation this implies that $[a',b']=[a-1,b-1]$. Moreover, $[a,b]$ is a longest segment in $\fm$.

    It follows that for $\fm''=\fm-[a,b]-[a',b']+[a+1,b]+[a'+1,b']$ we have $a+(a+1)+\fm'=\fn$ and $(a+1)+a+\fm'=\fm$, finishing the claim.
\end{proof}
Let $w=i_1\ldots i_l\in\Omega$ and $\fm\in \MsQ$. We call $w$ a \emph{descendant} of $\fm$ if it is obtained from $\fm$ in the following, recursive, way.
Write $\fm=[a_1,b_1]+\ldots+[a_k,b_k]$ and choose $\ain{i}{1}{k}$. Then choose $w'$ a descendant of $\fm'=[a_1,b_1]+\ldots+[a_i+1,b_i]+\ldots+[a_k,b_k]$. If $\fm'$ is empty, $w'$ is the empty word. Finally, $w$ is a descendant of $\fm$ if it is of the form $w=a_iw'$ for some $a_i$ and $w'$ as above.
\begin{lemma}\label{L:descword}
    Let $w\in \Omega$ and $\fm\in \MsQ$. Then $w$ is a descendant of $\fm$ if and only if $\fm_{gen}(w)\lp\fm$.
\end{lemma}
\begin{proof}
    We argue by induction on the length of $\fm$, the case of length $1$ being trivially true, and we use the notation empolyed above the lemma. First assume that $w$ is a descendant of $\fm$ and write $w=a_iw'$, $w'$ a descendant of $\fm'$, as above the lemma. By the induction hypothesis we have that $\fm_{gen}(w')\lp\fm'$ and hence $\fm_{gen}(w)\lp a_i+\fm'$. It thus suffices to show that $a_i+\fm'\lp\fm$, which follows quickly. If $[a_i+1,b_i]$ is the longest segment in $\fm'$ starting in $a_i+1\mod n$, we have in fact equality. Otherwise, let $[a_i+1,d]$ be the respective longest segment. Then $a_i+\fm'=\fm-[a_i+1,d]+[a_i,d]$. Note that we can perform the elementary operation $[a_i,b_i]+[a_i+1,d]\mapsto [a_i,d]+[a_i+1,b_i]$ on $\fm$, yielding $a_i+\fm'$, and thus proving the claim.

    For the other direction, assume that $\fm_{gen}(w)\lp\fm$ and write $w=cw',\, c\in I$. Then it follows from \cite[Lemma 6.3.3]{DroMod} that $\ain{c}{a_1}{a_k}$ and $\fm_{gen}(w')\lp \fm'$, where $\fm'$ is of the form $[a_1,b_1]+\ldots+[a_j+1,b_j]+\ldots+[a_k,b_k]$ for some $j$ with $a_j=c\mod n$. The claim follows immediately.
\end{proof}
\section{Whittaker models}
We follow the setup of \cite{BerZel}, (\emph{c.f.} \cite{KurMat} for the case of non-zero characteristic).
Let $\psi_R$ be an additive character $\psi_R\colon \Ff\ra R$. Moreover, we demand that $\psi_\ql$ is $\zl$-valued and $\psi_\ql\otimes_\zl\fl=\psi_\fl$. By abuse of notation, we will from now on write $\psi=\psi_R$.
We extend $\psi$ to $N_n$ by defining \[\psi(u)=\psi\left(\sum_{i=1}^{n-1}u_{i,i+1}\right).\]
We define the space of Whittaker functions of $G_n$ with respect to $\psi$ as \[\Wc(\psi)=\Wc_n(\psi)\coloneq\]\[ =\ind_{N_n}^{G_n}\psi=\{W\colon G_n\ra R\text{ loc. constant}:\, W(ug)=\psi(u)W(g)\text{ for all }u\in N_n,\, g\in G_n\},\]
on whom $G_n$ acts by right-translation.
We call $\pi\in \Rep_n$ of Whittaker type if \[\dim_R\Ho_{G_n}(\pi,\Wc(\psi))=1,\] in which case we denote the image of $\pi$ in $\Wc(\psi)$ by $\Wc(\pi,\psi)$. In particular, $\Wc(\pi,\psi)$ is socle irreducible and its unique irreducible subrepresentation appears with multiplicity $1$ in it. Moreover, if \[\Wc(\pi,\psi)\ra \Wc(\pi',\psi)\] is a non-zero map for a second representation $\pi'$ of Whittaker type, it is injective.
If $\pi$ is on top of that irreducible we call it \emph{generic}.
Let $W$ be a Whittaker function. We then denote the map $\widetilde{W}$ by \[g\mapsto \widetilde{W}(g)\coloneq W(w_n(g^{-1})^t).\]
Finally, if $\pi$ is of Whittaker type, we denote by \[{\Wc(\pi,\psi)}^\mathfrak{c}=\{\widetilde{W}(g):\, W\in \Wc(\pi,\psi)\}=\Wc({\pi}^\mathfrak{c},\psi^{-1}).\]
We recall that if $\pi_1$ and $\pi_2$ is Whittaker type, so is $\pi_1\times \pi_2$, see \cite{BerZel}, \cite{BerZelI}, (\emph{c.f.} also \cite[III.1.10]{VigBook}), and $\Wc(\pi_1\times\pi_2,\psi)=\Wc(\Wc(\pi_1,\psi)\times \Wc(\pi_2,\psi),\psi)$. 
We denote by $\Rep_{W,\psi,n}$ the set of finite length subrepresentations of $\Wc_n(\psi)$ which are of Whittaker type and set \[\Rep_{W,\psi}\coloneq \bigcup_{n\in\mathbb{N}}\Rep_{W,\psi,n}.\]  Then $\Rep_{W,\psi}$ can be equipped with an associative product \[*\colon \Rep_{W,\psi}\times \Rep_{W,\psi}\ra \Rep_{W,\psi},\, (\pi,\pi')\mapsto \Wc(\pi\times \pi',\psi).\]
Associativity is an easy consequence of the uniqueness of the equality $\Wc(\pi_1\times\pi_2,\psi)=\Wc(\Wc(\pi_1,\psi)\times \Wc(\pi_2,\psi),\psi)$.

If $\De=[a,b]_\rho$ is a segment, the representation $\langle\De\rangle$ is of Whittaker type if $l(\De)<e(\rho)$ by \cite[Remark 8.14]{MinSec} and every cuspidal representation is of Whittaker type.
\subsection{ Derivatives}
We recall the four exact functors in the definition of the Bernstein Zelevinsky derivatives, see \cite{BerZel}, \cite{BerZelI}, \emph{c.f.} also \cite[III.1]{VigBook}.
Recall also the groups $P_{n}$ from \Cref{S:notations} and write the unipotent subgroups as \[U_{n-1,1}\coloneq\left\{\begin{pmatrix}
    1&x\\&1
\end{pmatrix}\in G_n:\, x\in M_{n-1,1}\right\}.\]
\begin{enumerate}
    \item $\Psi^+\colon \Rep_{n-1}\ra \Rep(P_n)$, the extension by the trivial representation of $U_{n-1,1}$ and twisted by $\abs^{\frac{1}{2}}$.
    \item $\Psi^-\colon \Rep(P_n)\ra \Rep_{n-1}$, the $U_{n-1,1}$-coinvariants twisted by $\abs^{-\frac{1}{2}}$.
    \item $\Phi^+=\Ind_{P_{n-1}U_{n-1,1}}^{P_n}(-\otimes \psi)\colon \Rep(P_{n-1})\ra \Rep(P_n)$
    \item $\Phi^-\colon \Rep(P_n)\ra \Rep(P_{n-1})$, the $(U_{n-1,1},\psi)$-coinvariants twisted by $\abs^{-\frac{1}{2}}$.
\end{enumerate}
Let $\tau\in \Rep(P_n),\, \ain{k}{1}{n}$ and set $\tau_{(k)}\coloneq (\Phi^-)^{k-1}(\tau)$ and $\tau^{(k)}\coloneq \Psi^-(\tau_{(k)})$. For $\pi\in \Rep_n$ we set $\pi_{(k)}\coloneq (\restr{\pi}{P_n})_{(k)},\, \pi^{(k)}\coloneq (\restr{\pi}{P_n})^{(k)}$ and $\pi^{(0)}\coloneq\pi$.
\subsection{ Derivatives and Whittaker models}
We recall also the Kirillov model of a representation $\pi$ of Whittaker type. Namely, define it as the $P_n$-representation given by \[\Kr(\pi,\psi)\coloneq \{\restr{W}{P_n}:\, W\in \Wc(\pi,\psi)\}\subseteq \ind_{N_n}^{P_n}\psi.\]
\begin{theorem}[{\cite[4.3]{MattHelm}, \cite{MatMoss}, \cite{KurMatIII}}]\label{T:Kirillow}
Let $\pi\in \Rep_n$ be of Whittaker type.
    The map $W\mapsto \restr{W}{P_n}$ is injective on $\Wc(\pi,\psi)$.
\end{theorem}
Moreover, we have the following description of $(\Phi^-)^k$.
\begin{lemma}[{\cite[Proposition 1.3]{CogPia}}]\label{L:phirest}
    Let $\pi$ be of Whittaker type. Then we can identify $(\Phi-)^k(\Kr(\pi,\psi))$
    with the space
    \[\left\{p\mapsto \lvert\det(p)\lvert^{-\frac{k}{2}}W\begin{pmatrix}
        p&\\&1_k
    \end{pmatrix}:\, p\in P_{n-k},\, W\in \Wc(\pi,\psi)\right\}.\]
\end{lemma}
Note that the authors only prove the claim for $R=\CC$, however the same method works for arbitrary base fields. 
Finally, the description of the space $\Psi^-(\Phi^-)^k(\Kr(\pi,\psi))$ is a bit more tricky.
Let $\tau$ be a subrepresentation of $\Kr(\pi,\psi)^{(k)}$ with central character $\chi$.
Let $\sigma$ be the inverse image of $\tau$ in $\Kr(\pi,\psi)$. Define for $W\in \sigma$ and $g\in G_{n-k}$ the map \begin{equation}\label{E:lim}S(W)(g)=\lim_{z\ra 0}\lvert z\lvert^{\frac{k-n}{2}}\lvert\det(g)\lvert^{\frac{k}{2}}\chi^{-1}(z)W\begin{pmatrix}
    zg&\\&1_k
\end{pmatrix}.\end{equation}
Here $z\in \Ff^\times$ is seen as an element of $Z_{n-k}$ and the limit becomes stationary for $z$ small enough. 
\begin{prop}[{\cite[Proposition 1.7]{CogPia}, \cite[Corollary 2.1]{MatIII}}]\label{L:limit}
    The map $S\colon \sigma\mapsto \Wc(\psi)$ induces the non-zero, injective map
\[\overline{S}\colon \tau\hra \Wc(\psi).\]
\end{prop}
Whittaker models are a useful place to look for integral structures, as the next theorem shows.
\begin{theorem}[{\cite[Theorem 2]{Vign}}]\label{T:Whitint}
    For $\pi$ an integral representation of Whittaker type over $\ql$ of $G_n$, set
\[\Wc^{en}(\pi,\psi)=\{W\in \Wc(\pi,\psi):\, W(G_n)\subseteq \zl\}.\]
    If $\pi\in \Rep$ is integral and of Whittaker type, then $\Wc^{en}(\pi,\psi)$ is an integral structure of $\Wc(\pi,\psi)$.
\end{theorem}
We denote in this case
\[\ol{\Wc(\pi,\psi)}\coloneq \Wc^{en}(\pi,\psi)\otimes_\zl\fl.\]
\begin{lemma}\label{L:instrucwhit}
    Let $\pi_1,\ldots,\pi_k$ be integral representations of Whittaker type. Then the non-zero map
    \[\Wc(\pi_1,\psi)\times\ldots\times\Wc(\pi_n,\psi)\ra \Wc(\pi_1\times\ldots\times\pi_k,\psi)\] respects the integral structures and hence induces a non-zero map
    \[\ol{\Wc(\pi_1,\psi)}\times\ldots\times\ol{\Wc(\pi_n,\psi)}\ra \ol{\Wc(\pi_1\times\ldots\times\pi_k,\psi)}.\]
\end{lemma}
\begin{proof}
    See for example the proof of \cite[Theorem 2.26]{KurMat}.
\end{proof}
We recall the following useful properties of Whittaker models.
\begin{prop}[{\cite[Proposition 3.7]{KurMatIII}}]\label{L:derind}
Let $\pi_1\in \Rep_{n_1},\, \pi_2\in \Rep_{n_2}$
    \[\Wc(\pi_2,\psi)\subseteq \Wc(\pi_1\times \pi_2,\psi)^{(n_1)}.\]
\end{prop}
\begin{theorem}[{\cite[Theorem 3.10]{KurMatIII}}]\label{T:Matringe}
    Let $n\ge 2$ and $\tau$ be an a submodule of $\ind_{N_n}^{P_n}\psi$. If for $\ain{k}{1}{n-1}$, $\tau^{(k)}$ admits a central character, for any $W_0\in \Wc(\tau^{(k)},\psi)$ and $\phi\in C_c^\infty(\Ff^{n-k})$ there exists $W\in \tau$ such that for all $g\in G_{n-k}$ \[W\begin{pmatrix}
        g&\\&1_{k}
    \end{pmatrix}=W_0(g)\phi(\epsilon_{n-k}g)\lvert\det(g)\lvert^{\frac{k-1}{2}}.\]
\end{theorem}
The proofs of these two results are stated only for $R=\CC$, however their methods readily generalize to more general settings presented here.
\subsection{ Rankin Selberg L-factors}
We now recall the construction of Rankin Selberg L-factors as presented in \cite{KurMat}. For complex representations this goes back to the classical work of \cite{JacPiaSha}.
Let $\pi\in \Rep_n$ and $\pi'\in \Rep_m$ be representations of Whittaker type.
\begin{definition}
Let $W\in \Wc(\pi,\psi),\, W'\in \Wc(\pi',\psi^{-1})$ and $k\in \ZZ$.
\begin{enumerate}
    \item The case $n=m$: Let $\phi\in C_c^{\infty}(\Ff^n)$ and define
    \[c_k(W,W',\phi)\coloneq \int_{N_n\backslash G_n^k}W(g)W'(g)\phi(\epsilon_kg)\dd g\]
    and
    \[I(X,W,W,\phi)=\sum_{k\in \ZZ}c_k(W,W',\phi)X^k\in R(X).\]
    \item The case $n<m$: Let $\ain{j}{0}{n-m-1}$ and define
    \[c_k(W,W',j)\coloneq \int_{M_{j,m}}\int_{N_n\backslash G_n^k}W(g)W'\begin{pmatrix}
        g&&\\x&1_j&\\&&1_{n-m-j}
    \end{pmatrix}\dd g\dd x\]and
    \[I(X,W,W,j)=\sum_{k\in \ZZ}c_k(W,W',j)X^k\in R(X).\]
    \item The case $n>m$: Analogous to the case $n<m$.
\end{enumerate}
\end{definition}
Having defined this, we can now recall the definition of the $L$-factors.
\begin{theorem}[{\cite[Theorem 3.5]{KurMat}}]
    Let $\pi\in \Rep_n$ and $\pi'\in \Rep_m$ be representations of Whittaker type. 
    \begin{enumerate}
        \item The case $n=m$. The ideal spanned by $I(X,W,W',\phi)$, where we vary over $W\in \Wc(\pi,\psi),\, W'\in \Wc(\pi',\psi^{-1})$ and $\phi\in C_c^\infty(\Ff^n)$ is fractional and admits a generator $L(X,\pi,\pi')$ with $L(X,\pi,\pi')^{-1}\in R[X]$, which is normalized by demanding $L(0,\pi,\pi')=1$.
        \item The case $n\neq m$. Fix $\ain{j}{0}{\lvert n-m\lvert-1}$. The ideal spanned by $I(X,W,W',j)$, where we vary over $W\in \Wc(\pi,\psi)$ and $ W'\in \Wc(\pi',\psi^{-1})$ is fractional, independent of $j$ and admits a generator $L(X,\pi,\pi')$ with $L(X,\pi,\pi')^{-1}\in R[X]$, which is normalized by demanding $L(0,\pi,\pi')=1$.
    \end{enumerate}
\end{theorem}
In particular $L(X,\pi,\pi')=L(X,\pi',\pi)$.
As usual, these $L$-factors satisfy a functional equation, giving rise to an $\epsilon$-factor. For $\phi \in C_c^\infty(\Ff^n)$, we denote by $\hat{\phi}$ its Fourier-transform with respect to the character $\psi$.
\begin{theorem}[{\cite[Corollary 3.11]{KurMat}, \cite[Lemma 3.12]{KurMat}}]
    Let $\pi\in \Rep_n,\pi'\in \Rep_m$ be two representations of Whittaker type with central characters $c_\pi$ and $c_{\pi'}$ and $n\le m$. Let $W\in \Wc(\pi,\psi),\, W'\in \Wc(\pi',\psi^{-1}),$ $\phi\in C_c^\infty(\Ff^n)$ and $\ain{j}{0}{m-n -1}$.
Then there exists $\epsilon(X,\pi,\pi',\psi)\in R[X,X^{-1}]^\times$ such that the following holds.
\begin{enumerate}
    \item The case $n=m$.
    \[\frac{I(q^{-1}X^{-1},\Tilde{W}, \Tilde{W'},\hat{\phi})}{L(q^{-1},\pi\cc,{\pi'}\cc)}=c_{\pi'}(-1)^{n-1}\epsilon(X,\pi,\pi',\psi)\frac{I(X,W,W',\phi)}{L(X,\pi,\pi')}.\]
    \item The case $n<m$.
    \[\frac{I(q^{-1}X^{-1},\rho(w_{n,n-m})\Tilde{W}, \Tilde{W'},m-n-j-1)}{L(q^{-1},{\pi}^\mathfrak{c},{\pi'}^\mathfrak{c})}=\]\[=c_{\pi'}(-1)^{n-1}\epsilon(X,\pi,\pi',\psi)\frac{I(X,W,W',j)}{L(X,\pi,\pi')}.\]
\end{enumerate}
\end{theorem}
The last local factor, the $\gamma$-factor, is finally defined as follows.
\begin{definition}
    Let $\pi,\pi'\in \Rep$ be representations of Whittaker type. Then we define
    \[\gamma(X,\pi,\pi',\psi)\coloneq \epsilon(X,\pi,\pi',\psi)\frac{L(q^{-1}X^{-1},{\pi}^\mathfrak{c},{\pi'}^\mathfrak{c})}{L(X,\pi,\pi')}\in R(X).\]
\end{definition}
We now recall the most important properties of these local factors.
\begin{lemma}[{\cite[Theorem 3.13]{KurMat}}]\label{L:linc}
    Let $\pi,\pi',\pi''\in \Rep$ be representations of Whittaker type and $\tau$ a subrepresentation of $\pi$ of Whittaker type. Then
    \[\gamma(X,\pi,\pi',\psi)=\gamma(X,\tau,\pi',\psi)\] and \[L(X,\tau,\pi')^{-1}\lvert L(X,\pi,\pi')^{-1}.\]
    Moreover, we have the so-called inductivity relation
    \[\gamma(X,\pi\times\pi'',\pi',\psi)=\gamma(X,\pi,\pi',\psi)\gamma(X,\pi'',\pi',\psi).\]
\end{lemma}
Finally, let us remark how these factors interact with respect to reduction$\mod\ell$.
\begin{lemma}[{\cite[Theorem 4.1, §4.1]{KurMat}}]\label{L:lfredmodl}
    Let $\pi,\pi'\in \Rep$ be two integral representations of Whittaker type over $\ql$. Then $L(X,\pi,\pi'), \, \epsilon(X,\pi,\pi',\psi),\, \gamma(X,\pi,\pi',\psi)\in \zl(X)$ and
    \[\gamma(X,\ol{\Wc(\pi,\psi)},\ol{\Wc(\pi',\psi)},\psi)= \rl(\gamma(X,\pi,\pi',\psi)),\, \]\[L(X,\ol{\Wc(\pi,\psi)},\ol{\Wc(\pi',\psi)})^{-1}\lvert \rl(L(X,\pi,\pi')^{-1}).\]
\end{lemma}
\begin{lemma}\label{L:inclusion}
    Let $\pi\in \Rep_n,\pi'\in \Rep_m$ be two representation of Whittaker type and $\ain{k}{0}{n}$. Let $\tau$ be a subrepresentation of $\pi^{(k)}$ admitting a central character. Then
    \[L(X,\tau,\pi')^{-1}\lvert L(X,\pi,\pi')^{-1}.\]
\end{lemma}
\begin{proof}
    As explained in the proof of \cite[Lemma 4.6, Proposition 4.7]{KurMat}, see also \cite[Lemma 9.2]{JacPiaSha}, the claim holds true if one shows the following.
    For any $W_0\in \Wc(\tau,\psi)$ and $\phi\in C_c^\infty(\Ff^{n-k})$ there exists $W\in \Wc(\pi,\psi)$ such that for all $g\in G_{n-k}$ \[W\begin{pmatrix}
        g&\\&1_{n-k}
\end{pmatrix}=W_0(g)\phi(\epsilon_{n-k}g)\lvert\det(g)\lvert^{\frac{k-1}{2}}.\]
But by assumption on $\tau^{(k)}$, $\Wc(\tau^{(k)},\psi)$ admits a central character, hence this follows from \Cref{T:Matringe}.
\end{proof}
We finish by recalling the $L$-factors of $\Cc$-parameters of \cite{KurMatII}. For the moment we slightly extend this notation to aperiodic multisegments and to avoid confusion we will denote our parameters by $\Cc'$.
\begin{definition}
    Let $\rho,\rho'\in \cusp$. Then we set
    \[L(X,\Cc'(\rho),\Cc'(\rho'))\coloneq\]\[= \begin{cases}
        (1-(\chi(\varpi_\Ff)X)^{f(\rho)})^{-1}&\text{ if }\rho'\cong \chi\rho^\lor,\, \rho\in \scu,\, \chi\text{ an unramified character},\\
        1&\text{ otherwise.}
    \end{cases}\]
    Let $\De=[a,b]_\rho,\De'=[a',b']_\rho'$ be two segments. Then we set 
    \[L(X,\Cc'(\De),\Cc'(\De'))\coloneq \begin{cases}
        \prod_{i=a}^bL(X,\Cc'(\rho \lvert-\lvert^i),\Cc'(\rho' \lvert-\lvert^b))&\text{ if }l(\De)\le l(\De'),\\
        \prod_{i=a'}^{b'}L(X,\Cc'(\rho \lvert-\lvert^b),\Cc'(\rho' \lvert-\lvert^i))&\text{ if }l(\De)\ge l(\De').
    \end{cases}\]
    Finally, if $\fm,\fn\in \Ms$ of the form $\fm=\De_1+\ldots+\De_k,\, \fn=\De_1'+\ldots+\De_l'$ we set
    \[L(X,\Cc'(\fm),\Cc'(\fn))\coloneq\prod_{\substack{1\le i\le k \\ 1\le j\le l}} L(X,\Cc'(\De_i),\Cc'(\De_j')).\]
\end{definition}
\subsection{ Associative products and the map \texorpdfstring{$\sccp$}{Sc}}

For the rest of this section, we fix an additive character $\psi$ of $\Ff$.
Moreover, we also fix $\rho\in \scu$ and let $Q$ be the quiver of \Cref{S:genext} with $n=o(\rho)>1$. Fix a cuspidal support $\db=d_1\cdot [\rho\abs]+\ldots+d_{o(\rho)}[\rho\abs^{o(\rho)}]$ and let $\Ms(\rho)_\db$ be the set of multisegments with support $\db$. By abuse of notation we also denote by $\db$ the dimension vector $\db=(d_1,\ldots,d_{o(\rho)})$. We then recall the natural bijection
\[\Ms(\rho)_\db\leftrightarrow \{G_\db-\text{orbits in }E_\db\},\]
which by \Cref{P:oclos} respects the orders on each sides.
This bijection allows us to import the generic extension product to $\Ms(\rho)$, which by abuse of notation we also denote by $*$. 
This product has the following representation-theoretic interpretation.
By \Cref{T:serre}, there exists for $\fm\in \Ms^{ap}(\rho)$ representations $\rho_1,\ldots,\rho_k$ in $\ZZ[\rho]$ such that $\fm=\rho_1*\ldots*\rho_k$. We then set \[\mathcal{S}_{gen,\psi}(\fm)\coloneq\Wc(\rho_1\times\ldots\times \rho_k,\psi).\] We will see in a moment that this representation is independent of our choices.
To prove independence, we first note the following analoga of the degenerate Serre relations.
\begin{lemma}\label{L:serrep}
Let $\rho'\in \ZZ[\rho]$ such that $\rho'\ncong \rho\abs^{\pm1}$. Then
\[\Wc(\rho\times\rho',\psi)=\Wc(\rho'\times \rho,\psi).\]
If $o(\rho)>2$,
\[ \Wc(\rho\times\rho\times \rho\abs,\psi)= \Wc(\rho\times\rho\abs\times\rho,\psi),\]
\[\Wc(\rho\times\rho\abs\times \rho\abs,\psi)= \Wc(\rho\abs\times\rho\times\rho\abs,\psi).\]
If $o(\rho)=2$, 
\[\Wc(\rho\times\rho\abs\times \rho\times \rho,\psi)= \Wc(\rho\times \rho\times\rho\abs\times\rho,\psi).\]
\end{lemma}
\begin{proof}
Firstly, if $\rho'\ncong \rho\abs^{\pm1}$, $\rho\times\rho'$ is irreducible
\[\Wc(\rho\times\rho',\psi)\cong \rho\times\rho'\cong \rho'\times\rho\cong \Wc(\rho'\times \rho,\psi).\] and hence they are equal.
We only show the remaining claims for $o(\rho)>2$, the case $o(\rho)=2$ follows analogously.
We start with the equality $\Wc(\rho\times\rho\times \rho\abs,\psi)= \Wc(\rho\times\rho\abs\times\rho,\psi)$.
Note that by \cite[Proposition 7.17]{MinSec}, \[\rho\times\rho\times \rho\abs\sra \rho\times \langle[0,1]_\rho\rangle,\] which is irreducible \Cref{L:ordersubsubquot} and of Whittaker type. Thus the uniqueness of the Whittaker model forces \[\Wc(\rho\times\rho\times \rho\abs,\psi)\cong  \rho\times \langle[0,1]_\rho\rangle.\]
Similarly, \[\Wc(\rho\times\rho\abs\times \rho,\psi)\cong \langle[0,1]_\rho\rangle\times \rho.\] By the commutativity of parabolic induction on the Grothendieck group $\langle[0,1]_\rho\rangle\times \rho\cong \rho\times \langle[0,1]_\rho\rangle$ and the claim follows. The equality $\Wc(\rho\times\rho\abs\times \rho\abs,\psi)= \Wc(\rho\abs\times\rho\times\rho\abs,\psi)$ follows by an analogues argument as the equality $\Wc(\rho\times\rho\times \rho\abs,\psi)= \Wc(\rho\times\rho\abs\times\rho,\psi)$.
\end{proof}
As a corollary to the above lemma and \Cref{T:serre} we obtain the following. 
\begin{corollary}\label{C:assprod}
The representation
$\sccp(\fm)$ is independent of the particular sequence $\rho_1,\ldots,\rho_k$ with $\rho_1*\ldots*\rho_k=\fm$.
Moreover, the map $\sccp\colon\Ms(\rho)^{ap}\ra \Rep_W$ respects the respective products, \emph{i.e.} for $\fm_1,\fm_2\in \Ms(\rho)^{ap}$, we have
\[\Wc(\sccp(\fm_1)\times \sccp(\fm_2),\psi)=\sccp(\fm_1*\fm_2)\]
\end{corollary}
Next we define the map $\sccpu$ by setting for $\fm\in\Ms(\rho)$ not necessarily aperiodic
\[\sccpu(\fm)\coloneq \bigcup_{\substack{\fn\lp\fm,\\ \fn\text{ aperiodic}}}\sccp(\fn).\] Here the union is taken in the space of Whittaker functions.
\begin{prop}\label{L:sccorder}
    Let $\fn,\fm\in \Ms(\rho)$ with $\fn\lp\fm$. Then \[\sccpu(\fn)\subseteq \sccpu(\fm).\]
     Moreover, let $\fm'\in \Ms(\rho)$ such that $\fm'*\rho=\fm$. Then \[\Wc(\sccpu(\fm')\times\rho,\psi)\subseteq \sccpu(\fm).\] Similarly, if $\fm'\in \Ms(\rho)$ such that $\rho*\fm'=\fm$, then \[\Wc(\rho\times \sccpu(\fm'),\psi)\subseteq \sccpu(\fm).\]
\end{prop}
\begin{proof}
The first claim follows from the definition.
For the second claim we argue as follows.
We only treat the case $\fm'*\rho=\fm$ as the other one follows similarly. Then for $\fn'\lp\fm'$ we have by \Cref{L:genorder} that $\fn'*\rho\lp\fm$ and hence $\Wc(\sccpu(\fm')\times\rho,\psi)\subseteq \sccpu(\fm)$ by the first claim. 
\end{proof}
Finally for $\fm\in \Msi$, write $\fm=\fm_1+\ldots+\fm_k$ with $\fm_i\in \Msi(\rho_i),\, \rho_i\in \scu$ in pairwise different cuspidal lines. Then we set 
    \[\sccp(\fm) \coloneq \Wc(\sccp(\fm_1)\times\ldots\times \sccp(\fm_k),\psi),\,\sccpu(\fm)\coloneq \bigcup_{\substack{\fn\lp\fm,\\ \fn\text{ aperiodic}}}\sccp(\fn)\]
    and if $\fm'=\fm_1'+\ldots+\fm_k'\in \Msi,\, \fm_i'\in \Ms(\rho_i)$ is a second multisegment, we define  \[\fm*\fm'\coloneq \fm_1*\fm_1'+\ldots+\fm_k*\fm_k',\] extending the product $*$ to $*\colon \Msi\times \Msi\ra\Msi.$ 
\begin{lemma}\label{L:sccdiffsupseg}
Let $\fm=\fm_1+\ldots+\fm_k\in \Msi^{ap}$ as above. Then
    \[\Wc(\sccp(\fm),\psi)\cong \sccp(\fm_1)\times\ldots\times \sccp(\fm_k)\] and hence for $\fm'\in \Msi^{ap}$ a second multisegment
    \[\Wc(\sccp(\fm)\times\sccp(\fm'),\psi)=\sccp(\fm*\fm
').\] An analogous claim holds for $\sccpu$ and if $\fn\lp\fm$ then $\sccpu(\fn)\subseteq\sccpu(\fm)$.

\end{lemma}
\begin{proof}
    Note that there exists a non-zero map
    \[\sccp(\fm_1)\times\ldots\times \sccp(\fm_k)\ra \Wc(\sccp(\fm),\psi).\]
    Let $\pi$ be the unique irreducible subrepresentation of the latter representation. It is then enough to show that $\pi$ is the unique subrepresentation of the former representation. To see this note that by \Cref{T:msZ} and \Cref{L:irred}, the representation $\pi_1\times\ldots \times \pi_k$ is irreducible, where $\pi_i$ is the unique irreducible subrepresentation of $\sccp(\fm_i)$. Since $\pi_1\times\ldots \times \pi_k$ is generic, it has to be isomorphic to $\pi$. Thus $\pi$ is a subrepresentation of \[\sccp(\fm_1)\times\ldots\times \sccp(\fm_k).\]
    For the second point, note that if $\fm$ and $\fm'$ are in different cuspidal lines, then we just proved that $\Wc(\sccp(\fm)\times \sccp(\fm'),\psi)=\Wc(\sccp(\fm+\fm'),\psi)=\Wc(\sccp(\fm'+\fm),\psi)=\Wc(\sccp(\fm')\times \sccp(\fm),\psi)$. From this the second claim follows.
    Finally, if $\fn\lp\fm$ with $\fm_i,\fn_i\in \Ms(\rho_i)$ as above, we have by definition $\fn\lp\fm$ if and only if $\fn_i\lp\fm_i$ for all $i$. From the first point and \Cref{L:sccorder} it then follows that there exists an injection
    \[\sccpu(\fn)\hra\sccpu(\fm),\] which by the uniqueness of the Whittaker model must be an inclusion.

    The claim for $\sccpu$ follows the exact same pattern.
\end{proof}
\section{Standard modules}
For the rest of this section, we fix an additive character $\psi$.
\begin{definition}
    Two maps $\tc_\psi,\tc_{\psi^{-1}}\colon\Ms\ra \Rep$ are called
    of Whittaker type if they satisfy the following.
    \begin{enumerate}
        \item For each $\fm\in \Ms$, $\tc_\psi(\fm)$ is a representation of $G_{\deg(\fm)}$ and contained in $\Rep_{W,\psi}$
        \item $\tc_\psi(\fm)^{\mathfrak{c}}=\tc_{\psi^{-1}}(\fm^\lor)$.
                \item If $\fn\lp\fm$ then $\tc_\psi(\fn)\subseteq \tc_\psi(\fm)$.
    \end{enumerate}
    Since by (2), $\tc_{\psi^{-1}}$ is determined by $\tc_\psi$, we will usually omit $\tc_{\psi^{-1}}$.
    
    The map $\tc_\psi$ is called $L$-\emph{standard} if it is of Whittaker type and for all $\fn,\fm\in \Ms^{ap}$ \[L(X,\tc_\psi(\fn),\tc_{\psi^{-1}}(\fm))=L(X,\Cc'(\fn),\Cc'(\fm)).\]
   Finally, an $L$-standard map $\tc$ is called \emph{standard}
    if it moreover satisfies the following.
    \begin{enumerate}
        \item For each $\fm\in \Ms^{ap}$ there exists a non-zero map $\tc_\psi(\fm)\sra \langle\fm\rangle$. \item For each $\fm\in \Ms^{ap}$, $\dim_R\Ho_{\Delta G_{\deg(\fm)}}(\tc_\psi(\fm)\otimes \tc_{\psi^{-1}}(\fm^\lor),R)=1$.
        \item The multiplicity of $\langle\fm\rangle$ in $\tc_\psi(\fm)$ is $1$.
    \end{enumerate}
Replacing $\Ms$ by $\Msi$ in the above definition, we obtain the notion of $\square$-Whittaker type, $\square$-L-standard and $\square$-standard.
\end{definition} 
It is easy to see that if $\tc_\psi$ is standard and $\ell\neq 2$, the up to a scalar unique non-zero map $\tc_\psi(\fm)\ra \tc_{\psi^{-1}}(\fm^\lor)^\lor$ factors through $\langle \fm\rangle$.
Indeed, since $\ell\neq 2$, we have that $\langle\fm\rangle\cc\cong\langle\fm\rangle^\lor$ is a quotient of $\tc_{\psi^{-1}}(\fm^\lor)$, and since Aubert-duality commutes with taking duals, we have a non-zero map
\[\tc_\psi(\fm)\sra\langle\fm\rangle\hra \tc_{\psi^{-1}}(\fm^\lor)^\lor.\]
For $\tc$ a standard map, the representations $\tc_{\psi^{-1}}(\fm)$ with $\fm$ aperiodic are called \emph{standard modules}.
\subsection{ Standard modules over \texorpdfstring{$\ql$}{Ql}}\label{S:smql}
For this section we set $R=\ql$. For an arranged form $(\De_1,\ldots,\De_k)$ of a multisegment $\fm$ we denote \[T_\psi(\fm)\coloneq \Wc(\langle\De_1\rangle)\times\ldots\times\Wc(\langle\De_k\rangle,\psi),\] whose isomorphism type is independent of the arranged order, see for example \cite[Theorem 6.1]{ZelII}. If $\fm$ is moreover integral we write
\[T_\psi^{en}(\fm)\coloneq\Wc^{en}(\langle\De_1\rangle)\times\ldots\times\Wc^{en}(\langle\De_k\rangle,\psi).\]
\begin{theorem}\label{P:smql}
 The map $\Sc_\psi\colon \Msq\ra \Rep$ given by $\fm\mapsto \Wc(T(\fm),\psi)$ is standard.
\end{theorem}
\begin{proof}
    For the computation of $L$-factors see \cite[Theorem 8.2]{JacPiaSha}. The other properties are proved in \cite[§6, §7, §9]{ZelII}.
\end{proof}
We let $\Sc_\psi(\fm)^{en}$ be the subspace of $\zl$-valued function in $\Sc_\psi(\fm)$.
\begin{lemma}\label{L:genextql}
For all $\fm\in \Msq$, $\Sc_{gen,\psi}(\fm)=\Sc_\psi(\fm)$.
\end{lemma}
\begin{proof}
We argue by induction on $\deg(\fm)$. 
We first treat the chase where $\fm$ is one segment $[0,b]_\rho$. Then by \cite[Proposition 9.5]{ZelII} and Frobenius reciprocity $\rho\times \langle[1,b]_\rho\rangle\sra\langle[0,b]_\rho\rangle$, which is a generic representation by the uniqueness of the Whittaker model. Thus $\Wc(\rho\times \langle[1,b]_\rho\rangle,\psi)=\Wc(\langle[0,b]_\rho\rangle,\psi)$. By induction on $b$ the left hand side equals to $\Sc_{gen,\psi}([0,b]_\rho])$ by \Cref{L:sccorder}.

We come to the general case.
Write $\fm=\rho*\fm'$ as follows. Let $[0,b]_\rho, \rho\in\cusp$ be a longest segment in $\fm$ such all segments $\De$ in $\fm$ with $a_\rho(\De)=1$ satisfy $l(\De)<b$. Then we obtain with $\fm'=\fm-[0,b]+[1,b]_\rho$ the equality $\rho*\fm'=\fm$. Moreover, one can choose an arranged form $(\De_1,\ldots,\De_k)$ of $\fm'$ such that for $j<i$, where $i$ is the minimal $i$ such that $\De_i=[1,b]_\rho$, either $\De_j$ has cuspidal support not intersecting $\ZZ[\rho]$ or $a_\rho(\De_j)$ is not equal to $1$. In any case, $[0,0]+\De_j$ is unlinked for $j<i$.
We recall that by the uniqueness of the Whittaker model, we have for three representations $\pi_1,\pi_2,\pi_3$ of Whittaker type 
\[\Wc(\pi_1\times\pi_2\times\pi_3,\psi)=\Wc(\pi_1\times \Wc(\pi_2,\psi)\times\pi_3,\psi).\]
Thus by the case of a single segment
\[\Wc(\rho\times \langle\De_1\rangle\times\ldots\times\langle\De_k\rangle,\psi)=\]\[=\Wc(\langle\De_1\rangle\times\ldots\times\rho\times \langle\De_i\rangle\times\ldots\times\langle\De_k\rangle,\psi)=\Wc(\langle\De_1\rangle\times\ldots\times\Wc(\rho\times \langle\De_i\rangle,\psi)\times\ldots\times\langle\De_k\rangle,\psi)=\]\[=\Wc(\langle\De_1\rangle\times\ldots\times\langle{}^+\De_i\rangle\times\ldots\times\langle\De_k\rangle,\psi).\]
It is easy to see that $(\De_1,\ldots,{}^+\De_i,\ldots,\De_k)$ is an arranged form of $\fm$, hence the right hand side is equal to $\Sc_{\psi}(\fm)$. By induction and \Cref{L:sccorder} the left-hand side is equal to $\Sc_{gen,\psi}(\fm)$.
\end{proof}
\subsection{ Standard modules over \texorpdfstring{$\fl$}{Fl}}\label{S:smfl}
We come to our definition of standard modules over $\fl$.
Firstly, we recall from the introduction the following. If $\tfm$ is a lift of $\fm \in \Msl$, \emph{cf.} \Cref{S:intstruc}, we note in \Cref{T:Whitint} that $\Sc_\psi(\tfm)$ can be equipped with a natural integral structure, whose reduction $\mod\ell$ is again of Whittaker type and we denoted it by $\overline{\Sc_\psi(\tfm)}$. For $\fm\in\Msl$ we then defined in the introduction the intersection \[\Sc_\psi(\fm)\coloneq \bigcap_{\tfm\text{ lift of }\fm}\overline{\Sc_\psi(\tfm)}\] in the space of Whittaker functions.
Let us note that it can easily be seen that the representation is non-zero if $\fm\in \Msi$. Indeed, each $\overline{\Sc_\psi(\tfm)}$ contains with multiplicity one and as a unique subrepresentation the degenerate representation $\Z(\fs)$, where $\fs=\cus(\fm)$. Since $\dim_R\Ho_{G_n}(\Z(\fs),\Wc(\psi))=1$, $\Z(\fs)$ is a subrepresentation of $\Sc_\psi(\fm)$.
\begin{definition}
    Let $\tc_\psi$ a map of Whittaker type. We call $\tc$ \emph{extending} if for all $\fm\in \Msi$ and $\rho\in \scu$
    \[\Wc(\rho\times\tc_\psi(\fm),\psi)\subseteq\tc_\psi(\rho*\fm),\, \Wc(\tc_\psi(\fm)\times\rho,\psi)\subseteq\tc_\psi(\fm*\rho)\]
\end{definition}
\begin{lemma}\label{L:ext}
    The maps $\Scp$ and $\sccpu$ are extending.
\end{lemma}
\begin{proof}
For $\sccpu$ the claim is a consequence of \Cref{L:sccorder}. The claim for $\Scp$ in the case $R=\fl$ follows by noting that on the one hand, for every lift $\tfm'$ of $\rho*\fm$, one can find lifts $\trho$ and $\tfm$ of $\rho$ and $\fm$ such that $\trho*\tfm=\tfm'$ and vice versa. By \Cref{L:instrucwhit} the surjective, and hence nonzero, map $\Wc(\trho,\psi)\times \Scp(\tfm)\sra \Wc(\rho\times\Scp(\tfm),\psi)=\Wc(\Scp(\trho*\tfm),\psi)$
reduces to a non-zero map
\[\overline{\Wc(\trho,\psi)}\times \overline{\Scp(\tfm)}\ra\overline{\Wc(\Scp(\trho*\tfm),\psi)}.\]
Taking the intersection in the space of Whittaker models we obtain a non-zero map \[\rho\times \Scp(\fm)\ra \Scp(\rho*\fm).\]
To see that the map on the intersection is non-zero, it suffices to note that the map \[\overline{\Wc(\trho,\psi)}\times \overline{\Scp(\tfm)}\ra\overline{\Wc(\Scp(\trho*\tfm),\psi)}\] does not vanish on the subrepresentation $\overline{\Wc(\trho,\psi)}\times \Z(\cus(\fm))$, where  $\Z(\cus(\fm))$ is the unique irreducible subrepresentation of $\overline{\Scp(\tfm)}$, which is independent of the lift $\tfm$. 

By the uniqueness of the Whittaker model we thus have that $\Wc(\rho\times \Scp(\tfm),\psi)\subseteq \Scp(\rho\times\fm).$ The other inclusion follows by the same argument.
\end{proof}
\begin{corollary}\label{C:sminc}
    The map $\sccpu\colon\Msli\ra\Rep$ is of $\square$-Whittaker type. For $\fm\in \Msl$, $\sccp(\fm)\subseteq \Scp(\fm)$ and $\sccpu(\fm)\subseteq \Scp(\fm)$.
\end{corollary}
\begin{proof}
We start with proving properties (1) and (2) of being a $\square$-Whittaker type for $\sccp$, which proves them for $\sccpu$.
    The first claim is true by construction. The second claim follows from the following observation. Let $\fm\in\Msli$ and $\rho\in\scu$. Then it follows from the construction of $\fm_{gen}$ that $(\rho*\fm)^\lor=\fm^\lor*\rho^\lor$.
    In particular, if $\fm$ is of the form $\rho_1*\ldots*\rho_k$ for suitable $\rho_i\in \scu$, then $\fm^\lor=\rho_k^\lor*\ldots*\rho_1^\lor$. Note that moreover $\Msli$ is non-zero only if $\ell\neq 2$  and hence $\rho_i\cc\cong \rho_i^\lor$.
    We thus have that
    \[\mathcal{S}_{gen,\psi}(\fm)^{\mathfrak{c}}=\Wc(\rho_1\times\ldots\times\rho_k,\psi)\cc=\Wc(\rho_k\cc\times\ldots\times\rho_1\cc,\psi^{-1})=\mathcal{S}_{gen,\psi^{-1}}(\fm^\lor)\] and hence also the second claim follows.
    The third claim for $\mathcal{S}_{gen,\psi}$ follows from \Cref{L:genorder}.

Finally, to prove $\sccpu(\fm)\subseteq \Scp(\fm)$ it suffices to prove that for $\fn$ aperiodic with $\fn\lp\fm$ we have $\sccpu(\fn)\subseteq \Scp(\fm)$. Moreover, by \Cref{L:sccdiffsupseg} it suffices to assume that $\fm\in \Ms(\rho)$. Write $\fn=\rho_1*\ldots*\rho_k$. By the uniqueness of the Whittaker model it suffices to give a non-zero map $\rho_1\times\ldots\times\rho_k\ra \Scp(\fm)$. 
By \Cref{L:descword} and the Geometric Lemma, we have that $\rho_1\otimes\ldots\otimes \rho_k$ appears in 
$r_{\overline{P_\alpha}}(\overline{\Sc_\psi(\tfm)})$, for any lift $\tfm$ of $\fm$, as a subquotient. Since $\fm\in \Ms(\rho)$, considering the central characters implies that it appears as a subrepresentation, which in turn implies that there exists a non-zero map $\rho_1\times\ldots\times\rho_k\ra \overline{\Sc_\psi(\tfm)}$. Since $\rho_1\times\ldots\times\rho_k$ is of Whittaker type, $
\Wc(\rho_1\times\ldots\times\rho_k,\psi)$ is contained in the intersection of the $\overline{\Sc_\psi(\tfm)}$, implying the claim.
\end{proof}
\begin{rem}
    If $\fm$ is a banal multisegment, see \cite{MinSecII}, \emph{i.e.} for each cuspidal representation $\rho$, the cuspidal support of $\fm$ does not contain $\rho\abs^k$ for some $k$, the same argument as in \Cref{L:genextql} shows that $\sccp(\fm)$ and $\Scp(\fm)$ agree.
\end{rem}
\begin{lemma}\label{C:extendingproperty}
    Assume $\tc_\psi$ is an extending map of Whittaker type. Let $\fm\in \Msi$ and $\De=[a,b]_\rho$, $\rho\in \scu_m$, a segment in $\fm$. Then $\tc_\psi(\fm-\De+{}^-\De)\subseteq \tc_\psi(\fm)^{(m)}$.
\end{lemma}
\begin{proof}
We show that $\rho\abs^a*(\fm-\De+{}^-\De)\lp \fm$.
Let $\Gamma=[a-1,b']_\rho$ be the longest segment in $\fm-\De+{}^-\De$ with $a_\rho(\Gamma)=a-1$. If $\Gamma={}^-\De$ we have that the left side equals the right side. Otherwise, we have $b'>b$ and that the left side equals to
\[\fm-\De+{}^-\De-\Gamma+{}^+\Gamma,\] which can be obtained from $\fm$ via the elementary operation $\Gamma+\De\mapsto {}^+\Gamma+{}^-\De$.
 By the above observation and the properties of $\tc_\psi$, we obtain that 
    \[\Wc(\rho\abs^a\times \tc_\psi(\fm-\De+{}^-\De),\psi)\subseteq\tc_\psi(\fm).\]
The claim then follows from \Cref{L:derind}.
\end{proof}
\subsection{ L-factors of standard modules}
Fix an additive character $\psi$.
The goal of this section is to show that $\Sc_\psi$ and $\Sc_{gen,\psi}$ are $L$-standard.
\begin{theorem}\label{T:RSL}
    Let $\tc_\psi$ be an extending map of Whittaker type such that $\tc_\psi(\fm)\subseteq \Sc_\psi(\fm)$ for all $\fm\in \Ms$. Then $\tc_\psi$ is $L$-standard, \emph{i.e.} for $\fn,\fm\in \Msi^{ap}$ \[L(X,\tc_\psi(\fm),\tc_{\psi^{-1}}(\fn))=L(X,\Cc'(\fm),\Cc'(\fn)).\]
\end{theorem}
This has been achieved in \cite[Theorem 4.22]{KurMatII} for generic representations if $R$ is arbitrary. In the case $R=\ql$, this follows from the work of \cite{JacPiaSha} as mentioned in \Cref{S:smql}, yielding the following corollary of \Cref{L:lfredmodl} and \Cref{L:inclusion}.
\begin{corollary}\label{C:divisibility}
    Let $\fn,\fm\in \Ms$. Then \[L(X,\tc_{\psi}(\fm),\tc_{\psi^{-1}}(\fn))^{-1}\lvert L(X,\Cc'(\fm),\Cc'(\fn))^{-1}.\]
\end{corollary}
Before we come to the proof of \Cref{T:RSL}, we note the following useful lemmas, all of whom follow easily from the definitions.
\begin{lemma}\label{L:aperiodic}
    Let $\fm\in \Ms$ containing segments $\De=[a,b]_\rho,\, \De'=[a+1,b+1]_\rho$ for some $\rho\in \scu$. Let $\fn\in \Ms$ and let $\fn'$ be the sub-multisegment of $\fn$ consisting of segments of the form $[c,d]_{\rho'}$ with $\rho'\cong \chi\rho^\lor$ for some unramified character $\chi$ and $c-d=b-a$.
    Finally, set $\fm'=\fm-\De-\De'+[a,b+1]_\rho+[a+1,b]_\rho$.
    Then
    \[\frac{L(X,\Cc'(\fm),\Cc'(\fn))}{L(X,\Cc'(\fm'),\Cc'(\fn))}=\prod_{[c,d]_{\chi\rho^\lor}\in \fn'}(1-\chi(\varpi_\mathrm{F})q^{-a-d}X)^{-f(\rho)}.\]
\end{lemma}
\begin{lemma}\label{L:red}
    Let $\fm,\fn\in \Ms$ and $\De=[a,b]_\rho$ a segment of maximal length in $\fm$, where we assume $\rho\in \scu$. 
    Let $\fn'$ be the sub-multisegment of $\fn$ consisting of segments $[c,d]_{\rho'}$ with $\rho'\cong \chi\rho^\lor$ for some unramified character $\chi$ and $c-d\ge b-a$.
    Finally set $\fm'=\fm-\De+{}^-\De$. Then
    \[\frac{L(X,\Cc'(\fm),\Cc'(\fn))}{L(X,\Cc'(\fm'),\Cc'(\fn))}=\prod_{[c,d]_{\chi\rho^\lor}\in \fn'}(1-\chi(\varpi_\mathrm{F})q^{-a-d}X)^{-f(\rho)}.\]
\end{lemma}
\begin{lemma}\label{L:unit}
    Let $\fm,\fn\in \Ms$ and $\tc_\psi$ a map of Whittaker type such that $\tc_\psi(\fn)\subseteq \Sc_\psi(\fn)$, $\tc_\psi(\fm)\subseteq \Sc_\psi(\fm)$. Then
    \[\frac{L(X,\tc_{\psi}(\fm),\tc_{\psi^{-1}}(\fn))L(q^{-1}X^{-1},\Cc'(\fm^\lor),\Cc'(\fn^\lor))}{L(q^{-1}X^{-1},\tc_{\psi}(\fm^\lor),\tc_{\psi^{-1}}(\fn^\lor))L(X,\Cc'(\fm),\Cc'(\fn))}\] is a unit in $R[X,X^{-1}]$.
\end{lemma}
\begin{proof}
Combining \Cref{L:linc} and \Cref{L:lfredmodl} shows that that the above fraction is equal to
    \[\rl(\epsilon(X,\Sc_{\psi}(\tfm),\Sc_{\psi^{-1}}(\tfn),\psi))\epsilon(X,\tc_{\psi}(\fm),\tc_{\psi^{-1}}(\fn),\psi)^{-1}\] for suitable lifts $\tfm$ and $\tfn$ of $\fm$ and $\fn$. Since $\epsilon$-factors are units in $R[X,X^{-1}]$, the claim follows.
\end{proof}
\begin{lemma}\label{L:gcd}
Let $\alpha,\beta\in R$ such that $\alpha^f\neq \beta ^f$. Then \[\gcd(1-(\alpha X)^f,1-(\beta X)^f)=1.\]
\end{lemma}
We are now ready to give the proof of \Cref{T:RSL}.
\begin{proof}[Proof of \Cref{T:RSL}]
In fact we will show the slightly stronger statement that if $\fm,\fn\in \Ms$ with at least one of the aperiodic, then \[L(X,\tc_\psi(\fm),\tc_{\psi^{-1}}(\fn))=L(X,\Cc'(\fm),\Cc'(\fn)).\]We argue firstly by induction on $\deg(\fm)$ and $\deg(\fn)$ and for fixed $\deg(\fm)$ and $\deg(\fn)$ we argue moreover by induction on the order $\lp$ on the set of multisegments. The base case is \cite[Theorem 4.22]{KurMatII}.

By \Cref{C:divisibility} we know that there exists $P\in \fl[X]$ such that \[L(X,\tc_\psi(\fm),\tc_{\psi^{-1}}(\fn))^{-1}P(X)=L(X,\Cc'(\fm),\Cc'(\fn))^{-1}.\]
Moreover, let $\fm_b$ and $\fn_b$ be the banal parts of $\fm$ and $\fn$. By definition \[L(X,\Cc'(\fm),\Cc'(\fn))=L(X,\Cc'(\fm_b),\Cc'(\fn_b))\]
and by \Cref{L:linc} and \Cref{L:derind} and the base case, \[L(X,\tc_\psi(\fm_b),\tc_{\psi^{-1}}(\fn_b))= L(X,\tc_\psi(\fm),\tc_{\psi^{-1}}(\fn)),\] hence it is enough to treat the case where $\fm=\fm_b,\, \fn=\fn_b$.

We first assume without loss of generality that $\fm$ is not aperiodic, \emph{i.e.} it contains a multisegment \[[a,b]_\rho+\ldots +[a+o(\rho)-1,b+o(\rho)-1]_\rho.\]
Choose $\ain{i}{0}{o(\rho)-1}$ and using the notation of \Cref{L:aperiodic}, we set $\De=[a+i,b+i]_\rho,\, \De'=[a+i+1,b+i+1]_\rho$.
By assumption on, $\tc_\psi$,  $\tc_\psi(\fm')\subseteq\tc_\psi(\fm)$ and hence by induction on $\lp$ we obtain by \Cref{L:aperiodic} 
that 
\[P(X)\lvert \prod_{[c,d]_{\chi\rho^\lor}\in \fn:\, d-c=b-a}(1-\chi(\varpi_\mathrm{F})q^{-a-i-d}X)^{f(\rho)}.\]
Let $I_i=\{\chi(\varpi_\mathrm{F})q^{-a-i-d}: [c,d]_{\chi\rho^\lor}\in \fn:\, d-c=b-a\}$.
Assume that $P(X)$ is non-zero and let $(1-\alpha X)$ be one of its non-zero factors with $(1-\alpha X)\lvert 1-\chi(\varpi_\mathrm{F})q^{-a-i-d}X)^{f(\rho)}$ for some $i$ and $[c,d]_{\chi\rho^\lor}\in \fn:\, d-c=b-a$. 
Let $\ain{j}{1}{o(\rho)}$. Since \[(1-\alpha X)\lvert \prod_{[c,d]_{\chi\rho^\lor}\in \fn:\, d-c=b-a}(1-\chi(\varpi_\mathrm{F})q^{-a-j-d}X)^{f(\rho)},\] it follows that $(1-\alpha X)\lvert (1-\chi'(\varpi_\mathrm{F})q^{-a-j-d'}X)^{f(\rho)}$ for some $[c',d']_{\chi'\rho^\lor}\in \fn$. Therefore by \Cref{L:gcd},
\[\chi'(\varpi_\mathrm{F})^{f(\rho)}q^{f(\rho)(-a-j-d')}=\chi(\varpi_\mathrm{F})^{f(\rho)}q^{f(\rho)(-a-i-d)}.\]
It follows straightforwardly that this implies that for each $j$, $[c+j,d+j]_{\chi\rho^\lor}$ appears in $\fm$, contradicting the assumption on it being aperiodic.

Secondly, we assume that both $\fm$ and $\fn$ are aperiodic and moreover, without loss of generality the length of the longest segment in $\fm$ is greater or equal than the length of the longest segment in $\fn$. We will now use the notation of \Cref{L:red}, \emph{e.g.} $\De=[a,b]_\rho$ is a longest segment in $\fm$.
Now \Cref{C:extendingproperty} and the induction hypothesis in combination with \Cref{L:derind} and \Cref{L:inclusion} give
\[P(X)\lvert \prod_{[c,d]_{\chi\rho^\lor}\in \fn'}(1-\chi(\varpi_\mathrm{F})q^{-a-d}X)^{f(\rho)}.\]
If $\fn'$ is empty, we are done, hence we can assume that the longest segment in $\fm$ and the longest segment in $\fn$ have the same length.

Replacing $\fm$ and $\fn$ by $\fm^\lor$ and $\fn^\lor$ we obtain a polynomial $P^\lor(X)$ such that by \Cref{L:unit}
\[\frac{P(X)}{P^\lor(q^{-1}X^{-1})}=rX^k,\, r\in \fl,\, k\in \ZZ.\]
Applying the same reasoning as above to $P^\lor(X)$ we obtain 
\[P^\lor(X)\lvert \prod_{[c,d]_{\chi\rho^\lor}\in \fn'}(1-\chi(\varpi_\mathrm{F})^{-1}q^{b+c}X)^{f(\rho)}.\]
We now assume that $P(X)$ is not a constant, and has a zero at $(\chi(\varpi_\mathrm{F})q^{-a-d})^{-1}$ with $[c,d]_{\chi\rho^\lor}\in \fn'$. Then $P^\lor(q^{-1}X^{-1})$ has to have a zero also at $\chi(\varpi_\mathrm{F})q^{-a-d}$ implying that there exists $[c',d']_{\chi'\rho^\lor}\in \fn'$ such that 
\[\chi'(\varpi_\mathrm{F})q^{-b-c'+1}=\chi(\varpi_\mathrm{F})q^{-a-d}.\]
Thus $\chi'=\chi\abs^{-c'+1+c}$ and hence $[c+1,d+1]_{\chi\rho^\lor}\in \fn'$.

But now we can switch the roles of $\fm$ and $\fn$ as \[L(X,\tc_\psi(\fm),\tc_{\psi^{-1}}(\fn))=L(X,\tc_\psi(\fn),\tc_{\psi^{-1}}(\fm))\] and apply \Cref{L:red} with the longest segment $[c+1,d+1]_{\chi\rho^\lor}$, yielding that $[a-1,b-1]_\rho$ has to be a segment of $\fm$, since $(\chi(\varpi_\mathrm{F})q^{-a-d})^{-1}$ is a zero of $P(X)$.
Repeating this process we obtain that for all $i\in \ZZ_{\ge 0}$ the segment $[a-i,b-i]_\rho$ is contained in $\fm$, a contradiction to the assumption that $\fm$ is aperiodic.
\end{proof}
Recall now the $\Cc$-parameters of \cite{KurMatII}, \emph{i.e.}
the image of the injective map constructed in \cite{KurMatII}
\[\Cc\colon\Irr_n\ra\{\text{semi-simple Deligne }R-\text{representations of length }n\}.\]
For the precise definition we refer to \cite{KurMatII}. The right hand side of the above map is equipped with a tensor-product denoted by $\otimes_{ss}$ and
one can associate to two $\Cc$-parameters $\Cc(\pi)$ and $\Cc(\pi')$ the three local factors \[L(X,\Cc(\pi)\otimes_{ss}\Cc(\pi')),\,\epsilon(X,\Cc(\pi)\otimes_{ss}\Cc(\pi'),\psi),\, \gamma(X,\Cc(\pi)\otimes_{ss}\Cc(\pi'),\psi).\]
As a corollary to \Cref{T:RSL}, one can prove exactly as in \cite[§6.4.]{KurMatII} the following.
\begin{corollary}
    Let $\pi=\langle\fm\rangle,\,\pi'=\langle\fm'\rangle\in \Irr$. Then
    \[L(X,\Sc_\psi(\fm),\Sc_{\psi^{-1}}(\fn))=L(X,\Cc(\pi)\otimes_{ss}\Cc(\pi')),\]
    \[\epsilon(X,\Sc_\psi(\fm),\Sc_{\psi^{-1}}(\fn),\psi)=\epsilon(X,\Cc(\pi)\otimes_{ss}\Cc(\pi'),\psi),\]
    \[\gamma(X,\Sc_\psi(\fm),\Sc_{\psi^{-1}}(\fn),\psi)=\gamma(X,\Cc(\pi)\otimes_{ss}\Cc(\pi'),\psi).\]
If $\fm\in \Msi$, the same is true if one replaces $\Sc$ by $\sccpu$.
\end{corollary}
\subsection{ Quotients of standard modules}
Let $\fm\in \Msi$. We can define
\[J_\fm\colon \Sc_\psi(\fm)\otimes\Sc_{\psi^{-1}}(\fm^\lor)\otimes C_c^\infty(\Ff^n)\ra R\] given by \[(W,W',\phi)\mapsto\restr{L(X,\Sc_\psi(\fm),\Sc_{\psi^{-1}}(\fm^\lor))^{-1}I(X,W,W',\phi)}{X=1}.\]
Denote by $C_{c,0}^\infty(\Ff^n)$ the subspace of $C_c^\infty(\Ff^n)$ consisting of all function vanishing at $0$.
\begin{prop}\label{P:Jfm}
    The map $J_\fm$ vanishes for all $\phi\in C_{c,0}^\infty(\Ff^n)$. In particular, we obtain a non-zero map \[J_\fm\colon \Sc_{gen,\psi}(\fm)\otimes\Sc_{gen,\psi^{-1}}(\fm^\lor)\hra \Sc_\psi(\fm)\otimes\Sc_{\psi^{-1}}(\fm^\lor)\ra R\]
     given by \[W\otimes W'\mapsto J_\fm(W,W',\phi),\]
    where $\phi$ is some fixed element in $C_c^\infty(\Ff^n)$ such that $\phi(0)\neq 0$.

    Moreover, if $\tfm$ is a lift of $\fm$ then $J_\tfm$ restricts to a map
    \[J_{\tfm}\colon \Sc_\psi(\tfm)^{en}\otimes\Sc_{\psi^{-1}}(\tfm^\lor)^{en}\ra \zl\]
    whose reduction $\mod\ell$ equals to $J_\fm$.
\end{prop}
\begin{proof}
We start with the case $R=\ql$. By \cite[Proposition 4.6]{MatIV}, if $\tfm$ is totally unlinked
the map \[J_\tfm\colon \Sc_\psi(\tfm)\otimes\Sc_{\psi^{-1}}(\tfm^\lor)\otimes C_{c,0}^\infty(\Ff^n)\ra \ql\] vanishes. More generally, let $\tfm=\De_1+\ldots+\De_k$ with the segments in an arranged order, choose $s\in \mathbb{Z}^k$, and set
\[\Sc_{\psi,s}(\tfm)=\Wc(\Wc(\langle\De_1\rangle,\psi)\abs^{s_1}\times\ldots\times \Wc(\langle\De_k\rangle,\psi)\abs^{s_k},\psi)\] and \[\Sc_{\psi^{-1},s}(\tfm^\lor)=\Wc(\Wc(\langle\De_k^\lor\rangle,\psi^{-1})\abs^{-s_k}\times\ldots\times \Wc(\langle\De_1^\lor\rangle\abs^{-s_1},\psi^{-1}),\psi^{-1})\] and fix flat sections \[f_s\in \Sc_{\psi,s}(\tfm), \,f'_{s}\in \Sc_{\psi^{-1},s}(\tfm^\lor)\]in the sense of \cite[§3]{CogPia}.
We obtain for fixed $\phi\in C_{c,0}^\infty(\Ff^n)$ by \cite[Proposition 3.3]{JacPiaSha} a rational function over $\ql$ such that $P(X_1,\ldots,X_k)$ such that $J_{\fm_s}(f_s,f'_s,\phi)=P(q^{s_1},\ldots,q^{s_k})$. As we observed above, for all but finitely many $s$, $P(q^{s_1},\ldots,q^{s_k})=0$ and hence it vanishes everywhere. 

The case $R=\fl$ follows readily from the case $R=\ql$.
    Let $\tfm$ be any lift of $\fm$ to $\ql$. Then the map 
    \[J_\tfm\colon \Sc_\psi(\tfm)\otimes\Sc_{\psi^{-1}}(\tfm^\lor)\otimes C_{c,0}^\infty(\Ff^n)\ra \ql\] vanishes. Since $J_\tfm$ obviously respects the integral structures, it reduces to the map \[J_\fm\colon \Sc_\psi(\fm)\otimes\Sc_{\psi^{-1}}(\fm^\lor)\otimes C_{c,0}^\infty(\Ff^n)\ra \fl,\] which therefore also vanishes.
Moreover, if $\tfm$ is any multisegment over $\ql$, the map $\Sc_\psi(\tfm)\otimes\Sc_{\psi^{-1}}(\tfm^\lor) \ra \ql$ factors through $\langle\tfm\rangle\otimes \langle\tfm^\lor\rangle$ and in particular vanishes on $\Sc_\psi(\tfn)\otimes\Sc_{\psi^{-1}}(\tfn^\lor)$ for any $\tfn\lpp\tfm$. Hence $J_{\tfm}$ vanishes on $\Sc_\psi(\tfn)\otimes\Sc_{\psi^{-1}}(\tfm^\lor)\otimes C_c^\infty(\Ff^n)$.

We now argue that $J_\fm$ does not vanish on the restriction to $\sccp(\fm)\otimes\Sc_{gen,\psi^{-1}}(\fm^\lor)\otimes C_c^\infty(\Ff^n)$, which would prove the claim. For this it suffices to show that for $\fn\lpp\fm$ $J_\fm$ vanishes on $\sccpu(\fn)\otimes\Sc_{gen,\psi^{-1}}^\cup(\fm^\lor)\otimes C_c^\infty(\Ff^n)$. It suffices to show the claim for $\fn$ maximal, \emph{i.e.} that it is obtained from $\fm$ via one elementary operation. But then we can choose a lift $\tfm$ of $\fm$ and $\tfn$ of $\fn$ with $\tfn\lpp\tfm$ and the claim follows from the above observation and the fact that $\sccpu(\fn)\subseteq\Sc_\psi(\fn)\subseteq  \overline{\Sc_\psi(\tfn)}$, see \Cref{C:sminc}.
By an analogous argument we can also show that $J_\fm$ vanishes on $\sccpu(\tfm)\otimes\Sc_{gen,\psi^{-1}}^\cup(\tfn^\lor)\otimes C_c^\infty(\Ff^n)$ and hence it cannot vanish on
$\sccp(\fm)\otimes\Sc_{gen,\psi^{-1}}(\fm^\lor)\otimes C_c^\infty(\Ff^n)$. By the previous arguments we know that $J_\fm$ vanishes on $\sccp(\fm)\otimes\Sc_{gen,\psi^{-1}}(\fm^\lor)\otimes C_{c,0}^\infty(\Ff^n)$, which finishes the argument.
\end{proof}
\begin{theorem}\label{T:quotients}
Let $\fm\in\Msi^{ap}$. Then $\langle\fm\rangle$ is the unique irreducible quotient of $\Sc_{gen,\psi}(\fm)$.
\end{theorem}
\begin{proof}
Let $\fm\in \Msi^{ap}$.
We argue by induction on $\deg(\fm)$ and $\lp$. The base case is trivial.
Let $\pi=\langle\fn\rangle$ be a quotient of $\Sc_{gen,\psi}(\fm)$. Note that $\cus(\pi)=\cus(\fm)$.
Since $\pi$ has to be an irreducible subquotient of $\Scp(\fm)$, we have by \Cref{L:ordersubsubquot}
that $\fn\lp\fm$.

We write $\fm=\fm'*\rho, \rho\in \scu_m$ and $\fm'\in \Msi^{ap}$ and hence $\Sc_{gen,\psi}(\fm')\times\rho\sra\pi$ by \Cref{L:genorder}. In particular $\Sc_{gen,\psi}(\fm')\otimes \rho\hra \overline{r_{(\deg(\fm)-m,m)}}(\pi)$. By induction and \Cref{L:sccorder} we obtain that $\langle\fm'\rangle\otimes \rho$ appears in $\overline{r_{(\deg(\fm)-m,m)}}(\Scp(\fn))$. By the Geometric Lemma of Bernstein Zelevinsky, see \cite[Theorem 5.2]{BerZel}, \cite[Theorem 1.1]{ZelII}, it follows that $\fm'\lp\fn'$, where $\fn'$ is of the form $\fn'=\fn-[a,0]_\rho+[a,-1]_\rho$ for a suitable segment $[a,0]_\rho$. By the exact same argument as in \Cref{C:extendingproperty} we have $\fn'*\rho\lp \fn$. On the other hand by \Cref{L:genorder} $\fm=\fm'*\rho\lp\fn'*\rho\lp\fn$ and hence $\fn=\fm$.
Since $\Sc_{gen,\psi}(\fm)\subseteq \Scp(\fm)$, and $\langle\fm\rangle$ appears in the latter with multiplicity one by \Cref{T:multone}, it is also true for $\Sc_{gen,\psi}$.
\end{proof}
Let us state two corollaries to this result.
By abuse of notation we will also write $J_\fm\colon \Sc_\psi(\fm)\ra \Sc_{\psi^{-1}}(\fm^\lor)^\lor$ for the map obtained from \Cref{P:Jfm}.
As a consequence of this proposition, it restricts to a map $\Sc_{gen,\psi}(\fm)\ra \Sc_{gen,\psi^{-1}}(\fm^\lor)^\lor$, which factors through $\langle\fm\rangle$.
\begin{corollary}
    Let $\fm\in \Msi$. Then $\langle\fm\rangle$ appears in the image of $J_\fm$ as a quotient.
\end{corollary}
\begin{proof}
Let $I(\fm)$ denote the image of $J_\fm\colon \Sc_\psi(\fm)\ra \Sc_{\psi^{-1}}(\fm^\lor)^\lor$.
Let $\Sigma(\fm)$ be the kernel of the map  $\Sc_{gen,\psi}(\fm)\ra\langle\fm\rangle$. Then $\Sigma(\fm^\lor)=\Sigma(\fm)^{\mathfrak{c}}$. Since $\Sc_{gen,\psi}$ is $\square$-standard, the kernel of the map $J_\fm$ restricted to $\Sc_{gen,\psi}(\fm)$ is $\Sigma(\fm)$, thus $I(\fm)$ is a quotient of $\Pi(\fm)\coloneq \Sigma(\fm)\backslash \Sc_\psi(\fm)$. Moreover, we know that both $I(\fm)$ and $\Pi(\fm)$ contain $\langle\fm\rangle$ with multiplicity $1$ by \Cref{T:quotients} and $\Pi(\fm)$ contains it moreover as a subrepresentation.
Since $\Pi(\fm)^{\mathfrak{c}}\cong \Pi(\fm^\lor)$, it also follows that $\Pi(\fm)$ admits $\langle\fm\rangle$ as a quotient, and hence as a direct summand since it appears only with multiplicity $1$. Thus $\langle\fm\rangle$ is also a quotient of $I(\fm)$.
\end{proof}
The second corollary follows readily from \Cref{T:quotients}, \Cref{L:sccorder}, \Cref{L:ordersubsubquot} and \Cref{C:sminc}.
\begin{corollary}
    For $\fn,\fm\in \Msi$ we have $\sccpu(\fn)\subseteq \sccpu(\fm)$ if and only if
    $\fn\lp\fm$. We thus have an order preserving injection
    \[\sccpu\colon \Msi\hra\Rep_{W,\psi},\] respecting the products.
\end{corollary}
\bibliographystyle{abbrv}
\bibliography{References.bib}
\end{document}